\documentclass[graybox,natbib,nospthms,vecarrow]{SNmult}

\usepackage{type1cm}          
\usepackage{makeidx}          
\usepackage{graphicx}         
\usepackage{multicol}         
\usepackage[bottom]{footmisc} 
\usepackage{microtype}        

\usepackage{amsmath,amssymb,amsfonts,graphicx,bm,amsthm,mathrsfs}
\usepackage{soul}
\usepackage[colorlinks=true]{hyperref}
\definecolor{ultramarine}{rgb}{0 , 0, 0.3}
\definecolor{bulgarianrose}{rgb}{0.3, 0, 0}
\hypersetup{citecolor={ultramarine},
urlcolor={ultramarine},linkcolor={bulgarianrose}}
\usepackage{cleveref}
\usepackage{appendix}
\usepackage{pstricks}

\usepackage{enumitem}
\usepackage{latexsym}
\usepackage{stmaryrd}
\usepackage{bussproofs}
\usepackage{pstricks}
\usepackage{aliascnt}

\numberwithin{equation}{section} 
\newtheorem{theorem}{Theorem}[section]
\newaliascnt{lemma}{theorem}  
\newtheorem{lemma}[lemma]{Lemma}  
\aliascntresetthe{lemma}  
  
\newaliascnt{corollary}{theorem}  
\newtheorem{corollary}[corollary]{Corollary}  
\aliascntresetthe{corollary}  
 
\newaliascnt{proposition}{theorem}  
  
\aliascntresetthe{proposition}  
 
\newtheoremstyle{exampstyle}
{\topsep} 
{\topsep} 
{} 
{} 
{\bfseries} 
{.} 
{.5em} 
{} 
\theoremstyle{exampstyle}
\newaliascnt{fact}{theorem}  
  
\aliascntresetthe{fact}  
 
\newaliascnt{claim}{theorem}  
\newtheorem{claim}[claim]{Claim}  
\aliascntresetthe{claim}  
 
\newaliascnt{remark}{theorem}  
\newtheorem{remark}[remark]{Remark}  
\aliascntresetthe{remark}  
 
\newaliascnt{notation}{theorem}  
  
\aliascntresetthe{notation}  
 
\newaliascnt{example}{theorem}  
\newtheorem{example}[example]{Example}  
\aliascntresetthe{example}  
 
\newaliascnt{conjecture}{theorem}  
  
\aliascntresetthe{conjecture}  
 
\newaliascnt{question}{theorem}  
\newtheorem{question}[question]{Question}  
\aliascntresetthe{question}  
 
\newaliascnt{definition}{theorem}  
\newtheorem{definition}[definition]{Definition}  
\aliascntresetthe{definition}  
 
\newaliascnt{convention}{theorem}  
  
\aliascntresetthe{definition}

\newcommand{\shabox}{\,
                    \setlength{\unitlength}{1ex}
                    \begin{picture}(2,2)
                    \put(0,0){$\Box$}
                    \put(.45,-.05){\line(1,0){1.22}}
                    \put(.4,0){\line(1,0){1.22}}
                    \put(.35,.05){\line(1,0){1.22}}
                    \put(.3,.1){\line(1,0){1.22}}
                    \put(.25,.15){\line(1,0){1.22}}
                    \put(1.65,-.05){\line(0,1){1.22}}
                    \put(1.6,0){\line(0,1){1.22}}
                    \put(1.55,.05){\line(0,1){1.22}}
                    \put(1.5,.1){\line(0,1){1.22}}
                    \put(1.45,.15){\line(0,1){1.22}}
                    \end{picture}
                    \!}

\renewcommand{\qedsymbol} {\ensuremath{\shabox}}

\def\kcal{\mathcal{K}}

\def\nin{\not\in}

\def\PA{{\sf PA}}

\def\GL{{\sf GL}}

\def\R{\sqsubset}

\def\lcal{\mathcal{L}}

\newcommand{\Ax}{\AxiomC}
\newcommand{\UI}{\UnaryInfC}
\newcommand{\BI}{\BinaryInfC}

\newcommand{\LLa}{\LeftLabel}

\newcommand{\DP}{\DisplayProof}

\newcommand{\lr }{\leftrightarrow}

\newcommand{\myref}{biblio.bib}

\newcommand{\vdashsub}[1]{\vdash_{_{\!\!{#1}}}}

\def\nmodels{\not\models}
\newcommand{\myending}{\bibliographystyle{spbasic}%
\bibliography{\myref}%
\end{document}%
}
\newcommand{\myheading}[1]{%
\begin{document}
\title{#1}
\author{
Mojtaba Mojtahedi\thanks{\url{http://mmojtahedi.ir}}\\
Ghent University\\
Department of Mathematics: Analysis, Logic and Discrete Mathematics}
\maketitle
}

\usepackage{cleveref}
\usepackage{todonotes}

\newcommand{\giodo}[2]
{\todo[color=red, #1]{#2}}

\DeclareMathOperator{\dom}{\mathrm{dom}}
\DeclareMathOperator{\ran}{\mathrm{ran}}
\DeclareMathOperator{\set}{\mathrm{set}}

\newcommand{\rst}{{\restriction}}
\newcommand{\conc}{^{\smallfrown}}
\newcommand{\imp}{\rightarrow}
\newcommand{\biimp}{\mathrel{\leftrightarrow}}
\newcommand{\la}{\langle}

\newcommand{\On}{\mathsf{On}}

\newcommand{\rca}{\mathsf{RCA}_0}
\newcommand{\prso}{\mathsf{PRS}\omega}
\newcommand{\prsou}{\mathsf{PRS\omega U}}
\newcommand{\zfu}{\mathsf{ZFU}}
\newcommand{\atr}{\mathsf{ATR}_0}
\newcommand{\pica}{\Pi^1_1\mbox{-}\mathsf{CA}_0}

\newcommand{\kp}{\mathsf{KP}}
\newcommand{\kpu}{\mathsf{KPU}}

\newcommand{\ttt}{\mathsf{T}}

\DeclareMathOperator{\otp}{{o.t.}}
\DeclareMathOperator{\rank}{{rk}}
\DeclareMathOperator{\prim}{{Pr}}

\newcommand{\barprec}{\triangleleft}

\DeclareMathOperator{\tc}{TC}

\newcommand{\adma}{\mathbb{A}}
\newcommand{\langl}{\mathcal{L}}
\newcommand{\fragma}{\mathsf{L}_{\mathbb{A}}}
\newcommand{\LSigma}{\Sigma}
\newcommand{\Var}{\mathsf{Var}}

\newcommand{\fixp}{\text{Fix}}
\newcommand{\appr}{\text{App}}

\newcommand{\bew}{\mathsf{Bew}}
\newcommand{\prov}{\mathsf{Pr}}
\newcommand{\proov}{\mathsf{Proof}}

\newbox\gnBoxA
\newdimen\gnCornerHgt
\setbox\gnBoxA=\hbox{$\ulcorner$}
\global\gnCornerHgt=\ht\gnBoxA
\newdimen\gnArgHgt
\def\Godelnum #1{%
\setbox\gnBoxA=\hbox{$#1$}%
\gnArgHgt=\ht\gnBoxA%
\ifnum     \gnArgHgt<\gnCornerHgt \gnArgHgt=0pt%
\else \advance \gnArgHgt by -\gnCornerHgt%
\fi \raise\gnArgHgt\hbox{$\ulcorner$} \box\gnBoxA %
\raise\gnArgHgt\hbox{$\urcorner$}}
\usepackage{mathtools}
\newcommand{\IGL}[1]{\GL_{#1}}
\def\GLA{\IGL{\omega_1}}
\newcommand{\HGL}[1]{\mathcal{H}\GL_{#1}}
\newcommand{\hgl}{\mathcal{H}\GL}
\newcommand{\SGL}[1]{\mathcal{H}\GL_{#1}}
\def\SGLA{\SGL{\omega_1}}
\def\HGLA{\HGL{\omega_1}}
\def\HKA{\mathcal{H}\sfk_{\omega_1}}
\def\HKFA{\mathcal{H}\kfour_{\omega_1}}
\def\atom{{\sf atom}}
\newcommand{\lang}[1]{\lcal^{#1}}
\newcommand{\langz}[1]{\lcal_{#1}}
\newcommand{\langb}[1]{\lcal^\Box_{#1}}
\newcommand{\langn}[1]{\lcal^{[\,]}_{#1}}
\def\langza{\langz{\omega_1}}
\def\langa{\lang{\omega_1}}
\def\langba{\langb{\omega_1}}
\def\langna{\langn{\omega_1}}
\newcommand{\invert}[1]{\overline{#1}}
\def\aset{\mathcal{A}}
\def\CLA{\CL_{\omega_1}}
\def\SCLA{\mathcal{S}\CL_{\omega_1}}
\def\HCLA{\mathcal{H}\CL_{\omega_1}}
\def\CL{{\sf CL}}
\def\vdashna{\vdash_{_{\!\!\mathcal{N}\!\mathcal{A}}}}
\def\vdashsa{\vdash_{_{\!\!\mathcal{S}\!\mathcal{A}}}}
\def\vdashha{\vdash_{\scriptscriptstyle{\HCLA}}}
\def\natcla{\text{$\mathcal{N}\CLA$}}
\def\Con{{\sf Con}}
\newcommand{\con}[1]{\Con(#1)}
\newcommand{\decom}[2]{\bar\Gamma({#1},{#2})}
\newcommand{\mdecom}[2]{\Gamma({#1},{#2})}
\newcommand{\ldecom}[2]{\Gamma^\lozenge({#1},{#2})}
\newcommand{\bdecom}[2]{\Gamma^\Box({#1},{#2})}
\def\sfk{{\sf K}}
\def\kfour{{\sf K4}}
\def\seq{\mathcal{S}}
\newcommand{\inter}[1]{{#1}^*}
\newcommand{\seqsub}[2]{{#1}\{#2\}}
\newcommand{\gsub}[1]{\seqsub\Gamma{#1}}
\newcommand{\cut}[1]{{#1}^{^{{\text{cut}}}}}
\def\dskk{\mathcal{S}{\sf K}_{\omega_1}}
\def\dskkc{\mathcal{S}{\sf K}_{\omega_1}^{{{\text{cut}}}}}
\def\dsk{\mathcal{S}\kfour_{\omega_1}}
\def\dskc{\mathcal{S}{\sf K4}_{\omega_1}^{{{\text{cut}}}}}
\def\dgla{{\mathcal{S}\GL_{\omega_1}^{{\!\infty}}}}
\def\dglac{\mathcal{S}\GL_{\omega_1}^{{\!\infty,\text{cut}}}}
\def\ugam{{}}
\def\nsub{\Subset}
\newcommand{\occurrence}[3]{{#1}^{\scriptscriptstyle{{#2}}}_{\scriptscriptstyle{\!{#3}}}}
\newcommand{\osig}[1]{\occurrence{\sigma}{#1}{}}
\newcommand{\ftinysub}[1]{f_{\scriptscriptstyle{#1}}}
\def\fgam{\ftinysub{\Gamma}}
\def\Rt{\R^t}
\def\tR{\sqsupset^t}
\def\sft{{\sf T}}
\newcommand{\sub}[1]{{\sf sub}(#1)}
\def\root{{\sf root}}
\newcommand{\rankd}[1]{r({#1})}
\newcommand{\reachable}[3]{{#1}\R_{#2}{#3}}
\def\EV{\mathrel{\textup{E}}}
\def\SUP{S}
\def\kmodel{\kcal=(W,\R,\EV,\SUP)}

\begin{document}
\title*{Infinitary provability logic}
\titlerunning{Infinitary provability logic}
\author{Mojtaba Mojtahedi, Fedor Pakhomov, and Giovanni Sold\`a}
\authorrunning{M. Mojtahedi, F. Pakhomov, G. Sold\`a}
\institute{Mojtaba Mojtahedi \at Department of Mathematics: Analysis, Logic and Discrete Mathematics, Ghent University, Krijgslaan 281 S8, 9000 Ghent, Belgium, \email{Mojtaba.Mojtahedi@UGent.be}
\and Fedor Pakhomov \at Department of Mathematics: Analysis, Logic and Discrete Mathematics, Ghent University, Krijgslaan 281 S8, 9000 Ghent, Belgium; and Steklov Mathematical Institute of the Russian Academy of Sciences, Ulitsa Gubkina 8, Moscow 117966, Russia, \email{fedor.pakhomov@ugent.be}
\and Giovanni Sold\`a \at Department of Mathematics: Analysis, Logic and Discrete Mathematics, Ghent University, Krijgslaan 281 S8, 9000 Ghent, Belgium, \email{giovanni.a.solda@gmail.com}}

\maketitle

\abstract*{G\"odel-L\"ob provability logic $\GL$ is a propositional modal system that on one hand enjoys completeness with respect to conversely well-founded Kripke frames and on the other hand captures all modal principles about $\PA$-provability that are provable in $\PA$ itself. In the present paper we carry out an initial investigation into the question of what the infinitary counterpart of $\GL$  is. We develop a non-well-founded deep inference proof system $\dgla$ for the modal language with at most countably infinite conjunctions and disjunctions. We show that the calculus is sound and complete for well-founded transitive Kripke frames. Using Kripke-Platek set theory we develop an interpretation of the infinitary modal language in terms of infinitary provability over admissible sets. Then we show that a natural Hilber-style variant of infinitary $\GL$ is sound for this interpretation. We leave open, however, the question if $\dgla$ proves any additional theorems in comparison with the Hilbert-style calculus. Nevertheless, under certain conditions we do show that the infinitary provability logic arising from certain admissible sets lies between the set of theorems of the Hilbert-style calculus and the non-well-founded deep inference system.}

\abstract{G\"odel-L\"ob provability logic $\GL$ is a propositional modal system that on one hand enjoys completeness with respect to conversely well-founded Kripke frames and on the other hand captures all modal principles about $\PA$-provability that are provable in $\PA$ itself. In the present paper we carry out an initial investigation into the question of what the infinitary counterpart of $\GL$  is. We develop a non-well-founded deep inference proof system $\dgla$ for the modal language with at most countably infinite conjunctions and disjunctions. We show that the calculus is sound and complete for well-founded transitive Kripke frames. Using Kripke-Platek set theory we develop an interpretation of the infinitary modal language in terms of infinitary provability over admissible sets. Then we show that a natural Hilber-style variant of infinitary $\GL$ is sound for this interpretation. We leave open, however, the question if $\dgla$ proves any additional theorems in comparison with the Hilbert-style calculus. Nevertheless, under certain conditions we do show that the infinitary provability logic arising from certain admissible sets lies between the set of theorems of the Hilbert-style calculus and the non-well-founded deep inference system.
}

\section{Introduction}
Propositional modal logic $\GL$ is a normal modal logic extending the modal logic $\kfour$ of transitive Kripke frames by L\"ob's axiom $\Box (\Box \varphi \to \varphi)\to \Box \varphi$. Its significance arises primarily from its arithmetical semantics, where $\Box$ is interpreted as the formalized provability predicate for $\PA$. The celebrated result of R.~Solovay \citeyearpar{solovay1976provability} states that a modal formula $\varphi$ is a theorem of $\GL$ if and only if it expresses a universally $\PA$-provable principle about $\PA$-provability. $\GL$ possesses a well-behaved Kripke semantics \citep{segerberg1971essay}: it axiomatizes the class of transitive conversely well-founded Kripke frames and furthermore is complete with respect to the class of finite strict partial orders. On the structural proof-theory side of research, there is a traditional-style sequent calculus for $\GL$ enjoying a cut-elimination theorem \citep{avron1984modal}. Finally, there is a cyclic sequent calculus for $\GL$ that enjoys cut-elimination and is obtained from a traditional sequent calculus for $\kfour$ by allowing finite cyclic proofs with an occurrence of at least one $\Box$-rule along every loop \citep{shamkanov2014circular}. The question that this paper aims to study is what the appropriate infinitary counterpart of
the provability logic $\GL$ is.

Although infinitary modal logics certainly have received some
attention, they are severely understudied compared to their finitary counterparts.
Certain infinitary modal logics are known to be Kripke
complete \citep{radev1987infinitary,segerberg1994model}.
On the side of the structural proof theory of infinitary
modal logics,  the traditional-style sequent calculi are misaligned with Kripke semantics.
The situation here is that, depending on the particular details of the implementation,
sequent calculi for the countably infinitary counterpart of $K$
are either Kripke complete but lack cut-elimination, or have cut-elimination but are Kripke incomplete
\citep{minari2016some}. This issue is resolved by switching to more complex sequent structures. In particular, there is a complete cut-free calculus with tree-sequents
\citep{tanaka2001cut} and a complete cut-free labeled sequent calculus \citep{tesi2022proof}.

In the present paper we develop systems of deep inference $\dskk$, $\dsk$, and $\dgla$.
The systems $\dskk$ and $\dsk$ are systems of single-sided deep inference for the corresponding
infinitary modal logics; for the finitary case of the systems of one-sided deep inference the standard
references are \citep{kashima1994cut,marin2014label}, and for a broader overview of these and similar systems
we refer to the survey \citep{lellmann2024nested}. In fact our system $\dskk$ is rather close to the tree calculus
for the infinitary modal logic $\sfk$ considered in \cite{tanaka2001cut}: Tanaka's tree sequents are
in fact the same as two-sided nested sequents, the difference lying in the cosmetics of the presentation
rather than in the substantive content. The system $\dgla$, however, differs from those by being a
non-well-founded proof system.  We show that $\dskk$, $\dsk$, and $\dgla$ are Kripke complete.
An interesting difference between the setting of countably infinitary modal languages and the setting of
finitary modal languages that became transparent in the course of our investigation is that it is
completely unclear whether any other natural sort of calculus for infinitary $\GL$ can also attain
Kripke completeness: here by other natural calculi we mean sequent calculi where either proofs are
well-founded or only non-nested sequents are allowed.

Kripke-Platek set theory $\kp$ is a relatively weak subtheory of $\mathsf{ZFC}$ whose relevance
stems from the balance between the theory being strong enough to prove
various results within itself and at the same time weak enough to have many models
(\cite{barwise} is a standard source for the theory).
Initial developments of the system \citep{kripke1964transfinite,platek1966foundations} were largely motivated by
the goal of generalizing computability theory to infinite sets, with admissible sets (transitive models of $\kp$) serving as natural playgrounds, where  $\Sigma_1$-definable functions behave somewhat analogously to computable functions on naturals, $\Sigma_1$-definable sets behave somewhat analogously to c.e. sets of naturals, etc.
Another feature of $\kp$  that makes it important for our development of the
provability interpretation of infinitary modal logic is that it is classically known \citep{barwise1969infinitary} that $\kp$ is able to formalize
the notion of provability in infinitary first-order logic $L_{\infty,\omega}$. In particular, infinitary first-order logic restricted to countable admissible sets $A$ (countable transitive models of $\mathsf{KP}$)  satisfies natural completeness and compactness properties \citep{barwise}.

The study of what theories can prove about provability in themselves is a classical topic of research
in proof theory going back to the famous paper of Kurt G\"odel \citeyearpar{godel1931}, where
he proved the Incompleteness theorems.
Further developments led to a complete characterization of the propositional facts that $\mathsf{PA}$
can prove about its own provability, with L\"ob \citeyearpar{lob1955} isolating L\"ob's Theorem
and Solovay \citeyearpar{solovay1976provability} proving the arithmetical completeness of the logic $\mathsf{GL}$,
thus showing that the consequences of L\"ob's Theorem in the propositional modal language
fully exhaust what $\mathsf{PA}$ can prove about provability in itself.

The analogues of formalized provability-based semantics for $\mathsf{GL}$ have been studied for a number of finitary languages. Probably the most well-known direction is the logics of interpretability and conservativity  \citep{visser1990interpretability}. The systems that are conceptually closer to infinitary languages are the provability logics with quantifiers. Namely, predicate logics of provability \citep{Artemov1985,Vardanyan1986} and second order propositional provability logics \citep{Artemov1993}. 

In the present paper we develop a generalization of the theory of formalized provability to
the case of infinitary logic on admissible sets, following in the footsteps of the generalization of computability theory to admissible sets. Here, in place of the language of first-order Peano Arithmetic $\mathsf{PA}$, we consider for each admissible set $A$ the language $\langl[A]$ that consists of all infinitary first-order formulas $\varphi$ of the set-theoretic signature expanded by constants for the individual elements of $A$, such that $\varphi$ itself is an element of $A$.  As our base theory we take the infinitary theory $\mathsf{KP}[A]$ that extends $\mathsf{KP}$ just enough to be $\Sigma[A]$-complete ($\Sigma[A]$ is the $\langl[A]$-counterpart of the class of arithmetical $\Sigma_1$-formulas). This allows us to mimic the basic development of the theory of formalized provability: the Hilbert-Bernays-L\"ob derivability conditions, the Diagonal Lemma for infinitary formulas, L\"ob's Theorem, and the soundness of the infinitary Hilbert-style modal logic $\hgl[A]$.

Finally, we show that under certain rather restrictive conditions the technique of Solovay \citeyearpar{solovay1976provability} can be adapted to the infinitary case: for admissible sets in the setting with the axiom of constructibility, Kripke models admit a Solovay-style ``embedding'' into the infinitary provability logic. This gives the set of $\dgla$-theorems within $A$ as an upper bound on the provability logics of stable countable admissible sets. However, due to the fact that we do not know whether the $A$-fragment of $\dgla$ has the same theorems as $\hgl[A]$ or is larger, this, unlike the finitary case, does not allow us to completely characterize the provability logics  of those infinitary first-order theories.

\section{Infinitary classical logic}
In the present section we give a presentation of the classical results on the sequent calculus for infinitary
propositional logic. Although this sort of development has been known since the works of Sch\"utte \citeyearpar{Schutte1960} and
Tait \citeyearpar{tait}, we cover these results for the convenience of the reader: the developments of sequent calculi for the systems
of infinitary modal logics that we introduce later in the article are essentially built on top of these classical results.

The $\omega_1$-infinitary language for the propositional logic,  denoted as $\langza$, is defined as follows.
\begin{align*}
A::= \quad  p \mid  \invert p \mid \bigwedge X\mid \bigvee X.
\end{align*}
In the above definition, $p$ is considered to range over atomic formulas $\atom$,
and $X$ ranges over countable sets of $\langza$-formulas.
The set of atomic formulas $\atom$ is considered as an arbitrary set.
A \textit{literal} is either an atomic or its negation.
We define the negation  $\neg A$,  by recursion on $A\in\langza$:
$\neg p:= \invert p$, $\neg\invert p:= p$,  $\neg\bigwedge X:=\bigvee\neg X$
and $\neg\bigvee X:=\bigwedge\neg X$. In this notation we let
$\neg X:=\{\neg A: A\in X\}$.
Define $\bot:=\bigvee\emptyset$
and $\top:=\bigwedge\emptyset$.
Then  other connectives, e.g.~$\to$ and $\lr$
are defined in the usual way.
Then a sequent is a finite multiset of $\langza$-formulas.


Throughout this section, we assume that all formulas 	$A$ and  $B$ range over
$\langza$-formulas.
Also $X$ and $Y$ range over   countable sets of $\langza$-formulas.

An interpretation or model  $I$ for the language $\langza$ is a function
$I\colon \langza\longrightarrow \{0,1\}$, with following properties:
\begin{itemize}
\item $I(\invert p)=1-I(p)$.
\item $I(\bigwedge X)=\min I(X)$. Notice that this means in particular that
$I(\top)=1$.
\item $I(\bigvee X)=\max I(X)$. Notice that this means in particular that
$I(\bot)=0$.
\end{itemize}
Also define $\models A$ iff $I\models A$, for every interpretation $I$.

For a given set (multiset) $X$ of $\langza$-formulas, we define
$\sub X$ as the set of all subformulas of formulas appearing in $X$.
In particular, we assume that $X\subseteq\sub X$.

\subsection{Sequent calculus}
\textit{Sequents} are finite multisets of $\langza$-formulas. We use the capital Greek letters $\Gamma$ and $\Delta$
to range over sequents.
A sequent $\Gamma$ is called
\textit{trivial} if either $\top\in\Gamma$ or  for some atomic $p$ both $p$ and $\invert p$
belong to $\Gamma$.
We use the comma notation for unions of sequents,
namely $\Gamma,\Delta$ means $\Gamma\cup\Delta$. Also for single-element sequent,
we may skip curly braces and simply write $A$ for a single-element sequent $\{A\}$.
Semantially, we treat a sequent as disjunction of its elements and put $I\models \Gamma$ iff there is $A\in \Gamma$ such that $I\models A$.

The Tait-style \citep{tait}
sequent calculus  $\SCLA$ for infinitary classical non-modal logic includes the
following inference rules:
\\[5mm]
\begin{center}
\begin{tabular}{l l}
\multicolumn{2}{c}{
\Ax{}
\LLa{$\neg$}
\UI{$\Gamma,p,\invert p$}
\DP
}
\\[8mm]
\Ax{$\Gamma,\bigwedge X,A$, \;\; for all $A\in X$ }
\LLa{$\wedge$}
\UI{$\Gamma,\bigwedge X$}
\DP
\hspace*{3cm}
&
\Ax{$\Gamma,\bigvee X,A$}
\LLa{$\vee$}
\UI{$\Gamma,\bigvee X$}
\DP
\end{tabular}
\end{center}
In the $\neg$-axiom, $p$ is an atomic formula.  Also in the $\vee$-rule,
$A$ must be a formula in $X$. 
Notice that in the case where $X$ is empty, the $\wedge$-rule  becomes an axiom\footnote{A rule without any premises.}
(remember that $\top:=\bigwedge\emptyset$):
\begin{prooftree}
\Ax{}
\LLa{$\top$}
\UI{$\Gamma,\top$}
\end{prooftree}
It is known in the literature \citep{tait,barwise} that this system is sound and
complete for classical interpretations. Nevertheless,
as a warm-up for the completeness of similar calculi for modal logics,
and also for the sake of self-containedness, we prove the completeness here.

A partial $\SCLA$ proof-tree is a rooted tree\footnote{A tree is a partial order such that
for every node, the set of its ancestors is finite and linearly ordered.
It is rooted if it has a minimum element. In our terminology, trees are not necessarily rooted.} such that all of its nodes are labelled
by sequents, edges are labelled by one of the inference-rule labels  $\vee$ or
$\wedge$, and each node is derived from the set of its children by the
rule indicated in the label of the edges. Then an $\SCLA$ proof-tree is a
well-founded\footnote{No infinite ascending sequence of nodes exists.}
partial $\SCLA$ proof-tree such that all of its leaves are
labelled by trivial sequents (either $\top$ belongs to the leaf, or both $p$ and $\invert p$ belong to it for some atomic $p$).
Notice that $\SCLA$ proof-trees might be infinite, since the $\wedge$-rule might cause
infinitary branching in the tree. Also notice that for any partial proof tree,
if a label $\Delta$ appears above $\Gamma$, then $\Gamma\subseteq\Delta$.
We then define $\SCLA\vdash \Gamma$ for a sequent $\Gamma$, if there is
an $\SCLA$ proof-tree whose root is labelled by $\Gamma$.
\begin{remark}
    It is obvious that the proof system $\SCLA$ has the sub-formula property:
    for any partial $\SCLA$ proof-tree with $\Gamma$ at its root,
    all the formulas appearing above it are  also in $\sub\Gamma$.
\end{remark}

\begin{theorem}\label{SCLA-sound-comp}
$\SCLA$ is sound and complete for classical interpretations,
i.e.~for every sequent $\Gamma_0$, we have
$\SCLA\vdash  \Gamma_0$ iff $\models \Gamma_0$.
\end{theorem}
\begin{proof}
The soundness is by a straightforward transfinite induction on the proof-tree. So we only treat the completeness part.
Let $\SCLA\nvdash \Gamma_0$.
We find a model $I$ such that  $I\nmodels \Gamma_0$.

Before we continue with the proof, let us fix some notation.
Let $B_0,B_1,\ldots $ be an enumeration of all  formulas in $\sub {\Gamma_0}$ such that each formula
appears infinitely often. Notice that such an enumeration exists since $\sub {\Gamma_0}$ is countable.
Also let us fix a pairing function $\langle \cdot,\cdot\rangle\colon\mathbb{N}\times\mathbb{N}\to\mathbb{N}$
that is a bijection and satisfies $\langle i,j\rangle \geq  i$.

We first define, by recursion on $n$, a partial $\SCLA$-proof tree $T_n$ as follows.
$T_0$ is a single-node tree whose root is labelled by $\Gamma_0$.
Then assuming that $T_n$ is already defined and
$n+1=\langle i,j\rangle$, for every non-trivial sequent $\Gamma$
appearing as a label of a leaf $w$ in $T_n$ such that  $B_i\in\Gamma$,
we do the following:
\begin{itemize}
\item  If $B_{i}=\bigvee X$ and $B_j\in X$: then add $\Gamma,B_j$ above the
node $w$ in $T_n$ and label its edge by $\vee$.
\item If $B_i=\bigwedge X$:
then for any $B\in X$, add $\Gamma,B$
above $w$ labeling the edges by $\wedge$. Notice in particular that if $X=\emptyset$,
we do not add any node above $w$.
\item In other cases, do not add any node above $w$ in $T_n$.
\end{itemize}
Then, let $T^*$ be the union of all partial proofs $T_n$.
Since $\SCLA\nvdash A$,
$T^*$ cannot be an $\SCLA$ proof-tree. Therefore, either there exists an infinite
branch in the proof-tree labelled by $\Gamma_0\subseteq\Gamma_1\subseteq\ldots$, or there is a leaf sequent $\Gamma^*$ that is not an axiom.
In the first case we put $\Gamma^*:=\bigcup_{n=0}^\infty\Gamma_n$. In both cases, it is straightforward to verify that
the following properties hold for $\Gamma^*$:
\begin{itemize}
 \item $\Gamma_0\subseteq\Gamma^*$.
\item For any atomic $p$, either $p\nin\Gamma^*$ or $\invert p\nin\Gamma^*$.
\item If $\bigvee X\in\Gamma^*$, then $X\subseteq\Gamma^*$.
\item If $\bigwedge X\in\Gamma^*$, then $X\cap\Gamma^*$ is non-empty.
\end{itemize}
Then define $I\models p$ iff $\invert p\in \Gamma^*$. By induction on the construction
of $B$, using the properties listed above for $\Gamma^*$, one may easily show that
for every $B\in\Gamma^*$ we have $I\nmodels B$. Therefore $I\nmodels \Gamma_0$, as desired.
\end{proof}

As a corollary of the soundness and completeness theorems, we have the admissibility of cut for $\SCLA$.
The cut rule in our setting is the following inference rule:
\begin{prooftree}
    \Ax{$\Gamma,A$}
    \Ax{$\Gamma,\neg A$}
    \LLa{cut}
    \BI{$\Gamma$}
\end{prooftree}
\begin{corollary}\label{cut0}
    Cut is admissible in $\SCLA$.
\end{corollary}
\begin{proof}
    We reason contrapositively. Let $\SCLA\nvdash \Gamma$.
    Then by  completeness part of the \Cref{SCLA-sound-comp}, there is a model $I\nmodels \Gamma$. In this model, either $I\nmodels A$ or $I\nmodels\neg A$. Thus either $I\nmodels (\Gamma,A)$ or $I\nmodels(\Gamma,\neg A)$.
    Now the soundness part of the  \Cref{SCLA-sound-comp} implies either    $\SCLA\nvdash (\Gamma,A)$
    or $\SCLA\nvdash(\Gamma,\neg A)$, as desired.
\end{proof}
\subsection{Hilbert-style proof system $\HCLA$}\label{sec-HCLA}

We also can naturally define the Hilbert-style calculus for infinitary propositional logic.
Here we switch to the language that on top of infinitary conjunctions and disjunctions, also
allows to use negations and implications. This language is given by the inductive definition:
\[
A::= \quad  p \mid  \bigwedge X\mid \bigvee X \mid A\to B \mid \lnot A,
\]
where $p$ ranges over propositional variables, $A$ and $B$ range over formulas of the language, and $X$ over at most countable sets of formulas of the language.

The language of the Tait calculus from the previous subsection may be regarded as a fragment of this language by identifying the literal $\invert p$ with $\lnot p$. Conversely, every formula of the present language can be transformed into a formula of the Tait language by putting it into negation normal form, using the usual equivalences
\begin{gather*}
A\to B \equiv \lnot A\vee B,\\
\lnot\bigwedge X \equiv \bigvee\{\lnot A:A\in X\},\qquad
\lnot\bigvee X \equiv \bigwedge\{\lnot A:A\in X\}.
\end{gather*}
The semantics from the previous subsection extends to the present language in the standard way:
\[
I(\lnot A)=1-I(A),\qquad
I(A\to B)=\max(1-I(A),I(B)),
\]
together with the already given clauses for infinitary conjunctions and disjunctions.

The calculus is given by the following axioms and inference rules:
\begin{enumerate}
\item \label{Hilbert-start}$A\to(B\to A)$
\item $(A\to  (A\to B))\to (A\to B)$
\item $(A\to B)\to ((B\to C) \to (A\to C))$
\item $(A\to B)\to (\lnot B\to \lnot A)$
\item $A\to \lnot\lnot A$
\item $\lnot\lnot A \to A$
\item \label{Hilbert-end} $\displaystyle \frac{A\;\; A\to B}{B}$
\item \label{H_inf_ax_conj} $\bigwedge X\to A$, for each $A\in X$
\item \label{H_inf_ax_disj} $A\to\bigvee X$, for each $A\in X$
\item \label{H_inf_rule_conj} $\displaystyle \frac{A\to B\;\; \text{for each }B\in X}{A\to\bigwedge X}$
\item \label{H_inf_rule_disj} $\displaystyle \frac{A\to B\;\; \text{for each }A\in X}{\bigvee X\to B}$
\end{enumerate}
That is, we take the part of the Hilbert-Bernays calculus \citep[Section~III.3]{hubert1934grundlagen} for classical finitary propositional logic corresponding to the $\to$ and $\lnot$ connectives (\ref{Hilbert-start}.--\ref{Hilbert-end}.) and we add on top the natural rules and connectives corresponding to the infinitary conjunction and disjunction connectives. Proofs in $\HCLA$ are well-founded proof trees; because of the two infinitary rules they may be infinitely branching.

The soundness of $\HCLA$ can be verified in a straightforward manner: the finitary axioms are classically valid, the axioms for $\bigwedge$ and $\bigvee$ express the elimination of $\bigwedge$ and the introduction of $\bigvee$, and the infinitary rules express the introduction of $\bigwedge$ and the elimination of $\bigvee$, capturing the universal properties of countable conjunctions and disjunctions.

The calculus is also complete. Indeed, every proof in the Tait calculus $\SCLA$ can be transformed into a proof in $\HCLA$: a sequent is read as the disjunction of its formulas, initial sequents become classical tautologies, the $\vee$-rule is simulated by the axioms $A\to\bigvee X$, and the $\wedge$-rule is simulated by the infinitary rule introducing $\bigwedge X$. Hence, by the completeness of $\SCLA$, every valid formula in negation normal form is provable in $\HCLA$. Since the equivalences between formulas and their negation normal forms are derivable in $\HCLA$, the same holds for arbitrary formulas of the present language. Thus, for every formula $A$,
\[
\HCLA\vdash A
\quad\text{iff}\quad
\models A .
\]

\section{The infinitary modal logic}
In this section, we study the infinitary variants of the modal logics
$\sfk$, $\kfour$ and $\GL$. We find a deep-inference sequent calculus
for each of these systems, sound and complete for the
corresponding Kripke frames.
We also provide complete Hilbert-style axiomatizations for infinitary $\sfk$
and $\kfour$.
Nevertheless, a complete Hilbert-style axiomatization for infinitary $\GL$
remains open.

\subsection{Hilbert-style proof systems}\label{sec:hgl-omega}
We start by explaining how to set up the Hilbert-style calculi for the systems of infinitary modal logic. Since these systems are not the main focus of the present paper, we just give the general outline of the constructions.

The language of the Hilbert-style calculus  $\HKA$ extends the language of the propositional Hilbert-style calculus $\HCLA$ by an extra unary connective $\Box$.
It has the following axioms and rule:
\begin{enumerate}
\item All axiom schemes and rules of $\HCLA$.
\item Necessitation: \Ax{$A$}\UI{$\Box A$}\DP.
\item \label{GL_ax_first} Barcan's formula: $\bigwedge\Box X\to\Box\bigwedge X$.
\item $\Box (A\to B)\to(\Box A\to \Box B)$.
\end{enumerate}

Then $\HKFA$ is $\HKA$ plus the transitivity axiom:

\begin{enumerate}[resume]
\item $\Box A\to \Box \Box A$.
\end{enumerate}

Also $\HGLA$ is $\HKFA$ plus the L\"ob's axiom:

\begin{enumerate}[resume]
\item \label{GL_ax_last} $\Box(\Box A\to A)\to \Box A$.
\end{enumerate}

\begin{example}
Let $X:=\{\Box(\Box p_{n+1}\to p_n): n\in\omega\}$ and
$A:=\bigwedge X\to\Box p_0$. Then $A$ is valid in conversely well-founded
Kripke models. Notice that $A$ is derivable in $\HGLA$.
\end{example}

The usual Kripke semantics can be adapted to the case of infinitary modal languages in a straightforward way (we present a complete definition for the case of the language of one-sided sequent calculi in Section \ref{sec-gla-kripke}).

It is possible to obtain the following theorem:
\begin{theorem}\label{GLA-soundess-Kripke}
$\HKA$ is sound and complete for all Kripke models.
Also $\HKFA$ is sound and complete for transitive Kripke models.
Furthermore, $\HGLA$ is sound for transitive conversely well-founded Kripke models.
\end{theorem}
The soundness results are obtained by straightforward induction on the complexity of
the corresponding proofs.  For the completeness of $\HKA$ and $\HKFA$ one option is to reduce this
theorem to completeness of  nested sequent calculi $\dskk$ and $\dsk$, respectively that we prove latter \Cref{dgla-completeness}:
it is easy to simulate nested sequent calculus in a corresponding Hilbert-style system. Note that for $\HGLA$ this
sort of reduction is not avaliable (or at least we do not know if it is possible), since Kripke complete $\dgla$ is a non-well-founded proof system unlike $\dskk$ and $\dsk$. In fact Kripke completeness for the logic $\HKA$ was proved already in \cite{radev1987infinitary} and his technique, clearly can be adopted for the case of $\HKFA$. Finally, one can appropriately adjust the proof of Kripke-completeness via canonical models, as was already done in \cite{segerberg1994model} for certaine infinitary modal systems.

The completeness of $\HGLA$ with respect to the transitive conversely well-founded Kripke
models is unknown to us.
Nevertheless, we will provide a complete non-well-founded deep inference sequent
calculus which is sound and complete for transitive conversely
well-founded Kripke models (see \Cref{dgla-completeness}).

\begin{question}
Is the Hilbert-style system $\HGLA$ complete
for conversely well-founded transitive Kripke models?
\end{question}


\subsection{Nested sequents for modal logics}

The infinitary language $\langba$ for the one-sided sequent calculi for modal logics is defined as $\langza$ with two additional unary connectives
$\Box $ and $\lozenge$. We have the same notions of literals and negation for $\langba$ as we had for $\langza$, with
the addition that $\neg\Box A:=\lozenge\neg A$ and $\neg\lozenge A:=\Box\neg A$. Also $\top $ and $\bot$ are defined in the same way.
Throughout this section, we let $A$, $B$ and $C$ range over $\langba$-formulas. We also use $X$ and $Y$ to range over
countable sets of $\langba$-formulas.
Let us start with the Kripke semantics for this language.

A \textit{general nested sequent} is a multiset of \textit{general nested formulas}.
A general nested formula is either a formula in $\langba$ or of the form
$[\Gamma]$, with $\Gamma$ being a non-empty nested sequent. For the rest of this section,
we let $A$, $B$ and $C$ range over general nested formulas
and $\Gamma$, $\Delta$ and $\Theta$ range over general nested sequents.
A \textit{nested sequent} is a finite multiset of \textit{nested formulas},
and a nested formula is either a $\langba$-formula or of the form
$[\Gamma]$, with $\Gamma$ being a non-empty nested sequent.
We use comma notation as multiset union for sequents, i.e.~$\Gamma,\Delta$ means
$\Gamma\cup\Delta$. Also, for a singleton sequent we may skip curly braces and simply
write $A$ instead of the sequent $\{A\}$.

The  notation $\gsub{p}$ means that the following hold:
\begin{itemize}
\item $\Gamma$ is a nested sequent.
\item $p$ is an atomic formula that occurs in $\Gamma$ exactly once.
\item Either $\Gamma=p$, or $[p]$ occurs in $\Gamma$. In the first case,
$\gsub{p}$ is called a \textit{root occurrence};
otherwise it is an \textit{internal occurrence}.
\end{itemize}
Then $\gsub{\Delta}$ means the result of replacing $p$ by $\Delta$ in $\Gamma$.
Intuitively, $\gsub{\Delta}$ indicates a designated occurrence of
$\Delta$ in $\Gamma$.
Notice that this means that $p$ can occur
as $\Gamma$ itself, or immediately inside some bracket $[\ ]$.

One may also associate to general nested sequents their
syntax-trees\footnote{A tree is a partial order such that
for every node, the set of its ancestors is finite and linearly ordered.
It is rooted if it has a minimum element. In our terminology, trees are not necessarily rooted.}.
Given a general nested sequent $\Gamma$ and a general nested formula $A$,
we define the corresponding syntax-trees $T_\Gamma$ and $T_A$ recursively as follows:
\begin{itemize}
    \item $T_\Gamma$ is the disjoint union of all $T_A$ for $A\in\Gamma$ together with a fresh root beneath
    them labelled by $\root$. In particular this means that
    the multiplicities of the $T_A$'s are taken into account.
    \item $T_A$ for $A\in\langba$ is a single-node tree, labelled by $A$.
    \item $T_{[\Gamma]}$ is the disjoint union of all $T_A$ for $A\in\Gamma$ together with
        a fresh root beneath them labelled by $[\,]$. In particular this means
        that the multiplicities of the $T_A$'s are taken into account.
\end{itemize}
Notice that the corresponding syntax-tree of each general nested sequent
is a rooted conversely well-founded tree,
with all its leaves labelled by $\langba$-formulas, its root labelled by $\root$, and
other nodes labelled by $[\,]$.  We have the same observation for syntax-trees of general nested
formulas, with the following variation: all the nodes which are not leaves are labelled by $[\,]$.
An  occurrence $\gsub\Delta$ is then uniquely determined by a unique node
(a node which is not a leaf, and hence is labelled either by $\root$ or $[\,]$) in the syntax-tree.
The rooted tree generated by this node, after changing the label of its root to $\root$,
then corresponds to the general  nested sequent $\Delta$.

For simplicity of notation, we may annotate occurrences in general nested sequents
via lower case Greek letters    $\theta$ and $\delta$.
To emphasize that an occurrence $\delta$ is an occurrence in $\Gamma$,
we use the notation $\occurrence{\delta}{\Gamma}{}$.
Furthermore, if we want to indicate that $\delta$ is an occurrence of $\Delta$
in $\Gamma$, we use the notation $\occurrence\delta\Gamma\Delta$.
An occurrence $\occurrence\delta\Gamma{}$ is a
\textit{sub-occurrence} of another
occurrence $\occurrence\theta\Gamma{}$ iff the node corresponding to
$\occurrence\delta\Gamma{}$
in the syntax-tree of $\Gamma$
is a descendant of the one for $\occurrence\theta\Gamma{}$.
 A \textit{strict sub-occurrence} is the
irreflexive sub-occurrence relation: namely, $\occurrence\delta\Gamma{}$
is a strict sub-occurrence of $\occurrence\theta\Gamma{}$
iff $\occurrence\delta\Gamma{}\neq\occurrence\theta\Gamma{}$ and
$\occurrence\delta\Gamma{}$  is a sub-occurrence of $\occurrence\theta\Gamma{}$.
Finally, an \textit{immediate sub-occurrence} of an occurrence $\delta$
is an occurrence which, in terms of the syntax-tree, is a child of $\delta$.

By the definition of general nested sequents, there is no infinite
descending sequence of sub-occurrences of general nested sequents.
In other words, there is no infinite
nesting of brackets, just as we do not have infinite nesting of any
connective in the infinitary language. This observation is of particular importance when
we build a canonical model out of a general nested sequent, with its accessibility relation
corresponding to sub-occurrences of general nested sequents.
This observation allows us to conclude the converse well-foundedness of the Kripke model.

\begin{definition}
A labelled rooted tree\footnote{A tree is a partially-ordered set $(T,<)$ such that for every
$w\in T$, the set $\{u\in W: u\leq w\}$ is a finite linear ordering. A labelled tree is a tree with all of its nodes labelled.} is called a
\textit{non-well-founded nested tree} if all leaves are labelled by $\langba$-formulas, its root is
labelled by $\root$, and the other nodes are labelled by $[\,]$. A non-well-founded nested tree is called a
general nested tree if it is conversely well-founded, i.e.~no infinite  branch exists.
Finally, a finite  non-well-founded nested tree is simply called a nested tree.
\end{definition}

Observe that  up to isomorphism\footnote{An isomorphism between two labelled trees is a bijection
that preserves the accessibility relation and labels.},
(general) nested trees are syntax-trees of (general) nested sequents.
This justifies defining non-well-founded nested sequents as equivalence classes
(with respect to isomorphism of labelled trees) of
non-well-founded nested trees. Whenever it is clear from the context, we treat a non-well-founded nested
tree as a non-well-founded nested sequent.

In a similar way, we also define \textit{non-well-founded
nested formulas} as corresponding to non-well-founded nested trees, with the modification that the root
must also be labelled by $[\,]$ unless it is a leaf, in which case it must be labelled by a $\langba$-formula.

The notion of an occurrence in a non-well-founded nested sequent can be defined in a similar way as for nested sequents.
Namely, an occurrence in a non-well-founded nested sequent is a node in the corresponding
non-well-founded nested tree which is not a leaf.
This occurrence is an occurrence of the tree generated by the very same node\footnote{Notice that the label of the root
node of this generated tree might not be $\root$; in such a case we change its label to $\root$.}.

\subsection{Inference rules}
We have the following rules of inference, each corresponding to one of the connectives:
\\[5mm]
\begin{center}
\begin{tabular}{l l}
\multicolumn{2}{c}{
\Ax{}
\LLa{$\neg$}
\UI{$\gsub{\Delta,p,\neg p}$}
\DP}
\\[8mm]
\Ax{$\gsub{\Delta,\bigwedge X,A}$, \;\;   for all $A\in X$}
\LLa{$\wedge$}
\UI{$\gsub{\Delta,\bigwedge X}$}
\DP
\hspace*{3cm}
&
\Ax{$\gsub{\Delta,\bigvee X,A}$}
\LLa{$\vee$}
\UI{$\gsub{\Delta,\bigvee X}$}
\DP
\\[8mm]
\Ax{$\gsub{\Delta,[\Theta,A,\lozenge A],\lozenge A}$}
\LLa{$\lozenge$}
\UI{$\gsub{\Delta,[\Theta],\lozenge A}$}
\DP
&
\Ax{$\gsub{\Delta,[A],\Box A}$}
\LLa{$\Box$}
\UI{$\gsub{\Delta,\Box A}$}
\DP
\end{tabular}
\end{center}
\vspace{5mm}

In the $\vee$-rule, $A$ must be a formula in $X$.
Also in the $\neg$-rule, $p$ is atomic.
Notice that the specific case of the $\wedge$-rule when $X=\emptyset$
is actually the following rule without any premise:
\begin{prooftree}
\Ax{}
\LLa{$\top$}
\UI{$\gsub{\Delta,\top}$}
\end{prooftree}

In each of the above inference rules, we also define the \textit{active occurrence} of the inference rule in the following way:
\begin{itemize}
    \item In the $\neg$, $\wedge$, $\vee$ and $\Box$-rules, the active occurrence of the inference rule is the
    designated occurrence in the conclusion.
    \item In the case of the $\lozenge$-rule, the active occurrence of the inference rule is the occurrence of the
     nested sequent $\Theta$ in the conclusion.
\end{itemize}

The well-founded proofs using the above rules form a proof system for infinitary
$\kfour$. Furthermore, the non-well-founded version of the same rules
with a specific trace condition on infinite branches (see \cref{dgla} below) is the complete sequent
calculus for infinitary $\GL$. Finally, if we replace the $\lozenge$-rule with
its non-transitive version (see \cref{dskk} below), the well-founded proofs
form a complete sequent calculus for infinitary $\sfk$. More detailed definitions and the corresponding soundness and completeness theorems follow.
\subsection{The sequent calculus $\dsk$}
A \textit{partial proof-tree} is a rooted
tree
with the following properties:
\begin{itemize}
\item all nodes are labelled by sequents,
\item each node is derived from its parent(s) by one of the
above-mentioned inference rules.
\end{itemize}
Notice that the presence of the $\wedge$-rule
makes the tree potentially infinitely branching.
A  partial proof-tree is called \textit{well-founded}
	if no infinite ascending sequence exists.
A $\dsk$ \textit{proof-tree} is a well-founded partial  proof-tree
 in which all of its leaves  (nodes without any ancestors) are \textit{trivial}, i.e.~instances of either the $\neg$-axiom or the $\top$-axiom.
A nested sequent $\Gamma$ is derivable in $\dsk$ if there is a  $\dsk$ proof-tree
with $\Gamma$ at its root.
We then write $\dsk\vdash  A$ if $ A$ is derivable in $\dsk$.
We show that $\dsk$ is sound and complete for the class of transitive
Kripke models.

\subsection{The sequent calculus $\dskk$}\label{dskk}
We define $\dskk$ partial proof-trees  in the same way as
partial proof-trees, replacing the $\lozenge$-rule with the following inference rule:
\begin{prooftree}
\Ax{$\gsub{\Delta,[\Theta,A],\lozenge A}$}
\LLa{w$\lozenge$}
\UI{$\gsub{\Delta, [\Theta],\lozenge A}$}
\end{prooftree}
The active occurrence in this case is also the designated occurrence of $\Theta$ in the conclusion.
A $\dskk$ proof-tree is a well-founded partial $\dskk$ proof-tree
 in which all of its leaves  (nodes without any ancestors) are   instances of
 either the $\neg$-axiom or the $\top$-axiom.

A nested sequent $\Gamma$ is derivable in $\dskk$ if there is a
$\dskk$ proof-tree with $\Gamma$ at its root. We then write $\dskk\vdash  A$ if $ A$ is derivable in $\dskk$.
We will see that $\dskk$ is sound and complete for the class of all Kripke models.

Given a general nested formula $A$ (resp.~general nested sequent $\Gamma$), the set of all subformulas of $A$ (resp.~$\Gamma$), denoted by
$\sub A$ (resp.~$\sub\Gamma$), is defined in the usual way:
\begin{itemize}
    \item $\sub A:=\{A\}$ for $A$ being a literal.
    \item $\sub {\bigwedge X}:=\{\bigwedge X\}\cup\bigcup\{\sub A: A\in X\}$ and   $\sub {\bigvee X}:=\{\bigvee X\}\cup\bigcup\{\sub A: A\in X\}$.
    \item $\sub{\Box A}:=\{\Box A\}\cup\sub A$ and $\sub{\lozenge A}:=\{\lozenge A\}\cup\sub A$.
    \item $\sub\Gamma:=\bigcup\{\sub A: A\in \Gamma\}$. Notice that here we do not take multiplicity into account.
    \item $\sub{[\Gamma]}:=\{[\Gamma]\}\cup\sub\Gamma$.
\end{itemize}

\begin{remark}
    It is obvious that the inference rules introduced above have the sub-formula property:
    for any partial  ($\dskk$) proof-tree with $\Gamma$ at its root,
    all the $\langba$-formulas appearing above it are  also in $\sub\Gamma$.
\end{remark}

\subsection{Limit of infinite branches in partial proof trees}\label{sec-limit}

Let  $\Gamma_1$ be one of the premises used to derive $\Gamma_0$ via one of the $\dskk$- or $\dsk$-rules.
Then, given a syntax-tree $S_0$ of $\Gamma_0$, one may define the syntax-tree $S_1$ of $\Gamma_1$ in such a way that
$S_0\subseteq S_1$. More precisely, we define the syntax-tree $S_1$ of $\Gamma_1$ by cases on the
inference rule by which $\Gamma_0$ is concluded. In each case we assume that $\delta$ is the active occurrence in the inference rule and the node $w$ in $S_0$ corresponds to this occurrence.
\begin{itemize}
    \item $\wedge$: Let $\Gamma_0=\gsub{\Delta,\bigwedge X}$ and $\Gamma_1=\gsub{\Delta,\bigwedge X,A}$ for some $A\in X$.
    Then $S_1$ is obtained from $S_0$
    by adding a new child to $w$, labelled by $A$.
    \item $\vee$:  Let $\Gamma_0=\gsub{\Delta,\bigvee X}$ and $\Gamma_1=\gsub{\Delta,\bigvee X,A}$ for $A\in X$.
    Then  $S_1$  is obtained from  $S_0$
    by adding a new child to $w$, labelled by $A$.
    \item $\Box$:  Let $\Gamma_0=\gsub{\Delta,\Box A}$ and $\Gamma_1=\gsub{\Delta,[A],\Box A}$.
    Then    $S_1$  is obtained from $S_0$
    by adding a new child $u$ labelled by $[\,]$ to $w$, together with a child of $u$ labelled by $A$.
    \item w$\lozenge$:   Let $\Gamma_0=\gsub{\Delta,[\Theta],\lozenge A}$ and $\Gamma_1=\gsub{\Delta,[\Theta,A],\lozenge A}$.
    Then $S_1$ is obtained from $S_0$
    by adding a new child to $w$, labelled by $A$.
    \item $\lozenge$:   Let $\Gamma_0=\gsub{\Delta,[\Theta],\lozenge A}$ and $\Gamma_1=\gsub{\Delta,[\Theta,A,\lozenge A],\lozenge A}$.
    Then $S_1$ is obtained from $S_0$
    by adding two new children to $w$, labelled by $A$ and $\lozenge A$.
\end{itemize}
Therefore, in a partial proof-tree, we may assume that the syntax-tree $S_\Gamma$ of some general nested sequent
$\Gamma$ that appears above $\Delta$ in the partial proof-tree includes the syntax-tree $S_\Delta$ of $\Delta$.
This simple observation in particular implies that
any occurrence $\theta$ of some nested sequent in $\Delta$ continues to be an occurrence in $\Gamma$ as well.
Then for an infinite  branch
$\Gamma_0,\Gamma_1,\ldots$ in a partial proof-tree, we have
syntax-trees $S_i$ for the $\Gamma_i$ such that each $S_{i}$ is a subset of  $S_{i+1}$.  Therefore, we may define
the union $S^*:=\bigcup_{i\in\omega} S_i$.
Clearly, we have the following properties:
\begin{itemize}
    \item[i.]  $S^* $  is a non-well-founded nested tree.
Let $\Gamma^*$ denote the non-well-founded nested sequent corresponding to $S^*$.
\item[iii.] For any occurrence $\delta$  of some $\Delta^*$ in $\Gamma^*$ and any
$B\in\Delta^*\cap\langba$, there is some $n$ such that $\delta$ is also an occurrence
of some $\Delta_n$ in $\Gamma_n$ with $B\in\Delta_n$. Furthermore, $\delta$
remains an occurrence of some $\Delta_m$ in $\Gamma_m$  with $B\in\Delta_m$, for any $m>n$.
\end{itemize}

This justifies defining $\lim_{n\to\infty} \Gamma_n$ as the non-well-founded nested sequent corresponding to the union
of the corresponding syntax-trees of all the $\Gamma_n$.

\subsection{The sequent calculus $\dgla$}\label{dgla}
Let us first define the notion of nestedness for $\Box$-rules. Two instances of
$\Box$-rules are said to be \textit{nested}
if one appears above the other (not necessarily immediately) and
the occurrence of the active boxed formula in one $\Box$-rule
 is a sub-occurrence of the new bracket generated
by the other $\Box$-rule.
More precisely, let
\Ax{$\seqsub{\Gamma'}{\Theta',[A'],\Box A'}$}
\LLa{D1}
\UI{$\seqsub{\Gamma'}{\Theta',\Box A'}$}
\DP appears above
\Ax{$\gsub{\Theta,[A],\Box A}$}
\LLa{$D0$}
\UI{$\gsub{\Theta,\Box A}$}
\DP appear in some proof-tree.
Then we say that $D1$ is nested in $D0$ if the indicated occurrence of
$\Theta',\Box A'$ in the conclusion of $D1$ is a sub-occurrence of $[A]$ in $D0$.
Notice that the indicated occurrence of $[A]$ in
$D0$ might be expanded in the upper stages of the proof,
in the sense that $[A]$ may be replaced by $[A,\Xi]$.
Nevertheless, the occurrence itself remains.

A partial proof-tree is called a \textit{$\dgla$-proof} if
\begin{itemize}
\item all of its leaves  (nodes without any ancestors) are   instances of either the $\neg$-axiom or the $\top$-axiom,
\item any infinite branch in the proof-tree includes
infinitely many instances of $\Box$-rules, each of them nested in the previous one.
\end{itemize}

A nested sequent $\Gamma$ is derivable in $\dgla$ if there is a  $\dgla$ proof-tree
with $\Gamma$ at its root.
We then write $\Gamma\vdashsub{\dgla} A$ if there is some $\Delta\subseteq\Gamma$
such that   $\neg\Delta,A$ is derivable in $\dgla$.
We show that $\dgla$ is sound and complete for the class of transitive
conversely well-founded Kripke models.

\begin{remark}\label{rem-dgla}
    Given an infinite branch $\Gamma_0,\Gamma_1,\ldots$ in a partial proof tree,  let
    $\Gamma^*:=\lim_{n\to\infty}\Gamma_n$ be as defined in \cref{sec-limit}.
    Then $\Gamma^*$ is not a general nested sequent iff
    the branch includes infinitely many instances of $\Box$-rules, each of them nested in the previous one.
\end{remark}



\subsection{Kripke semantics}\label{sec-gla-kripke}
A Kripke model is a tuple $\kmodel$ in which $\R$ is a binary relation on
the set $W$ of possible worlds, $\SUP$ is a subset of the atomic formulas called the
support of $\kcal$,
and $\EV$ is a relation between possible worlds and atomic formulas in $\SUP$. We say that $\kcal$
supports a $\langba$-formula $A$ (resp.~a general nested sequent $\Gamma$)
if the set of all atomic formulas appearing in $A$ (resp.~$\Gamma$) is included in $\SUP$.
We then extend the relation $\EV$ to all $\langba$-formulas built up from atomic formulas in $\SUP$ as follows:
\begin{itemize}
\item $\kcal,w\models p$ iff $w\EV p$, for atomic $p\in\SUP$.
\item $\kcal,w\models\invert p$ iff $\kcal,w\nmodels p$, for atomic $p\in\SUP$.
\item $\kcal,w\models \bigwedge X$ iff $\kcal,w\models A$ for every $A\in X$.
\item $\kcal,w\models \bigvee X$ iff $\kcal,w\models A$ for some $A\in X$.
\item $\kcal,w\models \Box A$ iff for every $u\sqsupset w$ we have $\kcal,u\models A$.
\item $\kcal,w\models \lozenge A$ iff for some $u\sqsupset w$ we have $\kcal,u\models A$.
\end{itemize}
For a nested sequent $\Gamma$, we then
define:
\begin{itemize}
\item $\kcal,w\models \Gamma$ iff $\kcal,w\models A$ for \textit{some} $A\in\Gamma$.
\item $\kcal,w\models[\Gamma]$ iff for every $u\sqsupset w$ we have $\kcal,u\models \Gamma$.
\end{itemize}
$\kcal$ is called transitive (conversely well-founded) iff $(W,\R)$ is so.
For a Kripke model $\kcal$, the notation $\kcal\models A$ means that for every world $w$ in  $\kcal$ we have $\kcal,w\models A$.

\subsection{Soundness}

In this subsection, we show that
\begin{itemize}
\item $\dskk$ is  sound for arbitrary  Kripke models,
\item $\dsk$ is sound for transitive Kripke models,
\item $\dgla$ is sound for transitive and conversely well-founded Kripke models.
\end{itemize}

Let us first define some notions that are useful both here for the soundness theorems and later in the proof of
the admissibility of cut.
Given a nested sequent $\Gamma$ and an occurrence $\delta$ in it,
we define the \textit{rank} of $\delta$, denoted by $\rankd\delta$, as the height of $\delta$
in the syntax-tree of $\Gamma$.
Also, given a Kripke model $\kmodel$, $u,v\in W$ and $n\in\omega$,
we say that \textit{$v$ is $n$-reachable from $u$}, denoted by $\reachable u n v$, if there exists a sequence
$u=w_0\R w_1\R\ldots\R w_n=v$.  In particular, we assume that $\reachable u 0 v$ iff $u=v$.

\begin{lemma}\label{lem-cut}
    Let  $\delta$ be an occurrence of $\Delta$ in $\Gamma$, with $\Delta$ and $\Gamma$ being general nested sequents.
    Also let $\kmodel$ be a Kripke model, and $w\in W$. Then $\kcal,w\models \Gamma$ iff there is an
    occurrence $\occurrence \delta  \Gamma\Delta$ such that for every  $u\in W$ with
    $\reachable w {\rankd{\occurrence \delta{}{}}} u$ we have
    $\kcal,u\models\Delta$.
\end{lemma}
\begin{proof}
For the left-to-right direction, take $\delta$ to be the root occurrence in $\Gamma$.
For the other direction, let $\kcal,w\nmodels \Gamma$. By induction on
$n=\rankd{\occurrence\delta{}\Delta}$, we show that there is some $u$ with $\reachable w {n} u$
such that  $\kcal,u\nmodels\Delta$.
The case $n=0$ is obvious. So assume that $\rankd\delta>0$. Then $\occurrence \delta\Gamma\Delta$
is an immediate sub-occurrence of
the occurrence $\occurrence\theta \Gamma \Theta $. This means that $[\Delta]\in\Theta$.
By the induction hypothesis, there is some $v$ with $\reachable w{\rankd \theta}v$ such that $\kcal,v\nmodels\Theta$.
Since $[\Delta]\in\Theta$, this implies $\kcal,u\nmodels \Delta$ for some $v\R u$.
Therefore, by definition, we have $\reachable w {\rankd\theta+1} u$.  Since
$\rankd \delta=\rankd \theta+1$, we have the desired result.
\end{proof}

The soundness of a nested sequent calculus $\sft$ for a class $\mathfrak{M}$
of Kripke models is then defined as follows: for every nested sequent $\Gamma$ provable in $\sft$
and any Kripke model $\kcal\in\mathfrak{M}$ supporting $\Gamma$,
we have $\kcal,w\models\Gamma$ for every world $w$ in $\kcal$.
Let us begin with the soundness of $\dskk$ and $\dsk$:

\begin{theorem}\label{dsk-soundness}
$\dskk$  is sound for all Kripke models. Furthermore,
$\dsk$ is sound for transitive Kripke models.
\end{theorem}
\begin{proof}
We prove the soundness of $\dsk$ and leave the similar argument for $\dskk$ to the reader.
Let $\mathcal{P}$ be a $\dsk $-proof tree for $\Gamma_0$. Also assume that $\kmodel$
is a transitive model supporting $\Gamma_0$ such that $\kcal,w_0\nmodels \Gamma_0$. We show that there is an
infinite branch in the proof-tree, a contradiction.

We inductively define a sequence $\{\Gamma_n\}_{n\in\omega}$ of nested sequents
 and  functions $f_n$ with the following properties:
 \begin{enumerate}
 \item $\Gamma_0,\Gamma_1,\Gamma_2,\ldots, \Gamma_n$ is a branch in the proof-tree.
\item The domain of $f_n$ is the set of all occurrences
in $\Gamma_n$ and its co-domain is $W$.
\item $\kcal,f_n(\delta)\nmodels\Delta$,
for every occurrence $\delta$ of $\Delta$ in $\Gamma_n$.
 \end{enumerate}

Let us now define $\Gamma_n$ and $f_n$ with the above properties.
We only provide the definitions and leave the straightforward arguments for showing
the required properties 1--3 to the reader.
As the induction hypothesis, assume that we have already defined $f_m$ and $\Gamma_m$ for
every $m<n$, satisfying properties 1--3.
We have the following cases:
\begin{itemize}
\item   $n=0$:  
Given an occurrence $\delta$ of $\Delta$ in $\Gamma_0$, by \Cref{lem-cut}, there is some $u\sqsupseteq w_0$ such that
$\kcal,u\nmodels \Delta$. Define $f_n(\delta):=u$ for some such $u$.
\item $n>0$ and $\Gamma_{n-1}$ is derived by the $\neg$-rule:
Let $\Gamma_{n-1}=\gsub{\Delta,B,\neg B}$ with $\delta$ being the designated occurrence
of $\Delta,B,\neg B$ in $\Gamma_{n-1}$. Then by IH3
(i.e.~item 3 of the induction hypothesis) we have
$\kcal,f_{n-1}(\delta)\nmodels \Delta,B,\neg B$. Therefore,
$\kcal,f_{n-1}(\delta)\nmodels  B $ and $\kcal,f_{n-1}(\delta)\nmodels  \neg B $,
a contradiction. This shows that this case is impossible.
\item $n>0$ and $\Gamma_{n-1}$ is derived by the $\vee$-rule:
define $\Gamma_n$ to be the sequent just above $\Gamma_{n-1}$.
Then $\Gamma_n$ and $\Gamma_{n-1}$ share the same set of occurrences.
Therefore, for any occurrence $\delta$ in $\Gamma_{n-1}$, we define
$f_n(\occurrence\delta{\Gamma_n}{}):=f_{n-1}(\delta)$.
\item $n>0$ and $\Gamma_{n-1}$ is derived by the $\wedge$-rule:
let $\Gamma_{n-1}=\gsub{\Theta,\bigwedge\Delta}$, with $\delta$
being the designated occurrence
of $\Theta,\bigwedge\Delta$ in $\Gamma_{n-1}$. By IH3 we have
$\kcal,f_{n-1}(\delta)\nmodels \bigwedge\Delta$. Therefore, there is some $B\in\Delta$
such that $\kcal,f_{n-1}(\delta)\nmodels B$.
Define $\Gamma_n:=\gsub{\Theta,B,\bigwedge\Delta}$.
Then $\Gamma_n$ and $\Gamma_{n-1}$ share the same set of occurrences.
Therefore, for any occurrence $\delta$ in $\Gamma_{n-1}$, we define
$f_n(\occurrence\delta{\Gamma_n}{}):=f_{n-1}(\delta)$.
\item   $n>0$ and $\Gamma_{n-1}$ is derived by the $\lozenge$-rule:
define $\Gamma_n$ to be the sequent just above $\Gamma_{n-1}$.
Then $\Gamma_n$ and $\Gamma_{n-1}$ share the same set of occurrences.
Therefore, for any occurrence $\delta$ in $\Gamma_{n-1}$, we define
$f_n(\occurrence\delta{\Gamma_n}{}):=f_{n-1}(\delta)$.
\item $n>0$ and $\Gamma_{n-1}$ is derived by the $\Box$-rule:
then $\Gamma_{n-1}=\seqsub{\Gamma_{n-1}}{\Theta,\Box B}$
is derived from  $\Gamma_n:=\seqsub{\Gamma_{n-1}}{\Theta,[B],\Box B}$.
Let $\theta$ be the occurrence of $\Theta,\Box B$ in $\Gamma_{n-1}$.
The set of occurrences in
$\Gamma_n$ has exactly one additional occurrence which is not in $\Gamma_{n-1}$
(this is the only case in which a new occurrence is introduced),
namely $\delta$, the occurrence of the newly added bracket $[B]$ in $\Gamma_n$.
By the induction hypothesis (item 3), we have $\kcal,f_{n-1}(\theta)\nmodels \Box B$.
Therefore, there is some $u\in W$ with $f_{n-1}(\theta)\R u$. Define $f_n(\delta):=u$.
For all occurrences $\tau$ in $\Gamma_{n-1}$, define
$f_n(\occurrence\tau{\Gamma_n}{}):= f_{n-1}(\tau)$.
\qedhere
\end{itemize}
\end{proof}
\begin{theorem}\label{dgla-soundness}
$\dgla$ is sound for transitive conversely well-founded Kripke models.
\end{theorem}
\begin{proof}
Let $\mathcal{P}$ be a $\dgla $-proof tree for $\Gamma_0$. Also assume that $\kmodel$
supports $\Gamma_0$, that $\R$ is transitive, and that $\kcal,w_0\nmodels \Gamma_0$.
We show that there is an infinite sequence $ w_0\R w_1\R\ldots$. Therefore, $(W,\R)$ is not conversely well-founded.
Notice that this implies the desired soundness.

We define recursively a sequence $\{\Gamma_n\}_{n\in\omega}$ of nested sequents
 and  functions $f_n$ with the following properties:
 \begin{enumerate}
 \item $\Gamma_0,\Gamma_1,\Gamma_2,\ldots, \Gamma_n$ is a branch in the proof-tree.
\item The domain of $f_n$ is the set of all occurrences
in $\Gamma_n$ and its co-domain is $W$.
\item $\kcal,f_n(\delta)\nmodels\Delta$,
for every occurrence $\delta$ of $\Delta$ in $\Gamma_n$.
\item If $\theta$ is a strict sub-occurrence of $\delta$,
then $f_n(\delta)\R f_n(\theta)$.
\item The family of functions
$\{f_m\}_{m\leq n}$ is
compatible: if
$\delta$ is an occurrence in $\Gamma_m$ and $m<n$, then
$f_m(\occurrence\delta{}{})=f_{n}(\occurrence\delta{\Gamma_n}{})$.
Notice that, since $\Gamma_n$ appears above $\Gamma_m$ in the proof-tree,
every occurrence $\delta $ in $\Gamma_m$
can also be considered as an occurrence in $\Gamma_n$, denoted by
$\occurrence\delta{\Gamma_n}{}$.
 \end{enumerate}

Exactly the same definitions as in the proof of \Cref{dsk-soundness} work here.
 Since  $\{\Gamma_n\}_{n\in\omega}$ is an infinite branch in a $\dgla$-proof tree,
it must include infinitely many
nested  $\Box$-rules. Namely, let $I$ be an infinite set of indices
such that
$\{\Ax{$\Gamma_i$}\UI{$\Gamma_{i-1}$} \DP \}_{i\in I}$
is an infinite sequence of nested $\Box$-rules.
Hence, for each $i\in I$, there are $A_i$ and $\Theta_i$  such that 
$\Gamma_i=\seqsub{\Gamma_i}{\Theta_i,[A_i],\Box A_i}$ is the nested sequent
in the premise of the $\Box$-rule, with $\delta_i$ being the occurrence of
$A_i$ in the bracket.  Therefore, for every $i,j\in I$ with $j>i$,
$\delta_{j}$ is a strict sub-occurrence of $\occurrence{\delta_i}{\Gamma_{j}}{}$.
Define $w_{i}:=f_{i}(\delta_{i})$.
By items 3 and 4, for every
$i,j\in I$ with $i<j$, we have $w_i\R w_j$.
Given that $I$ is infinite, we have the desired infinite ascending sequence of worlds.
\end{proof}

\subsection{Completeness}

We define the following notions of cosaturation for a
non-well-founded nested sequent $\Gamma$:
\begin{itemize}
\item
\textit{$\neg$-cosaturated}: for every atomic formula $p$ and every $\Delta$ with
$\Gamma=\gsub{\Delta}$, either $p\nin \Delta$ or $\invert p\nin\Delta$.
\item
\textit{$\vee$-cosaturated}: for every $\Delta$  occurring in
$\Gamma$ with $\bigvee\Theta\in\Delta$, we have $\Theta\subseteq \Delta$.
\item
\textit{$\wedge$-cosaturated}: for every $\Delta$ occurring in
$\Gamma$ with $\bigwedge\Theta\in\Delta$, there is some $B\in\Theta\cap\Delta$. 
Notice that in particular this implies that $\top\nin\Delta$.
\item
w\textit{$\lozenge$-cosaturated}: for every $\Delta$ occurring in
$\Gamma$ with
$[\Theta],\lozenge B\in \Delta$, we have  $B\in\Theta$.
\item
\textit{$\Box$-cosaturated}:
for every $\Delta$ occurring in
$\Gamma$ with
$\Box B\in \Delta$, there is some $\Theta$ such that
$[\Theta]\in\Delta$ and $B\in\Theta$.
\item
\textit{weakly cosaturated}: all of the above items hold.
\item \textit{$\lozenge$-cosaturated}: for every $\Delta$ occurring in
$\Gamma$ with
$[\Theta],\lozenge B\in \Delta$, we have  $B,\lozenge B\in\Theta$.
Notice that this condition is stronger than w$\lozenge$-cosaturation.
\item \textit{cosaturated}: all of the above items hold.
\end{itemize}

Given a weakly cosaturated  $\Gamma$, we define a canonical model
$\kcal_\Gamma:=(W,\R,\EV,\SUP)$ corresponding to
$\Gamma$ as follows:
\begin{itemize}
\item We define $W$  to consist of all
occurrences in $\Gamma$.
\item Define $ \delta\R \theta$ iff $\theta$ is an immediate strict sub-occurrence of
$\delta$.
This in particular implies that if $\delta$ is the occurrence of $\Delta$
and $\theta$ is the occurrence of $\Theta$, then   $[\Theta]\in\Delta$.
\item $\SUP$ is the set of atomic subformulas of $\Gamma$.
\item $ \Delta\EV p$ iff $ p\nin\Delta$, for every $p\in\SUP$.
\end{itemize}
We then define $\kcal_\Gamma^t$ as the transitive closure of $\kcal_\Gamma$.
More precisely,
$\kcal_\Gamma^t:=(W,\Rt,\EV,\SUP)$, where $\Rt$ is the transitive closure of $\R$.
In other words,
$ \delta\Rt \theta$ iff $\theta$ is a strict sub-occurrence of
$\delta$. It is obvious that both $\kcal_\Gamma$ and $\kcal_\Gamma^t$ support every subformula of $\Gamma$.

\begin{lemma}[Truth]\label{saturation-dgla}
For every  weakly cosaturated non-well-founded nested sequent $\Gamma$,
every occurrence $\delta$ of $\Delta$ in $\Gamma$, and every
$B\in \Delta\cap\langba$,
we have $\kcal_\Gamma,\delta\nmodels B$.
Furthermore, if   $\Gamma$ is also cosaturated, then we have
$\kcal^t_\Gamma,\delta\nmodels B$.
\end{lemma}
\begin{proof}
We prove this by induction on the complexity of $B$.
This means that by the induction hypothesis,
for every subformula $C$ of $B$ and every occurrence $\delta$ of $\Delta$ in $\Gamma$
with $C\in \Delta$, we have $\kcal_\Gamma,\delta\nmodels C$.
Let $B\in\Delta\cap\langba$ and let $\delta$ be an occurrence of $\Delta$ in $\Gamma$.
We need to show that
$\kcal_\Gamma,\delta\nmodels B$.
We have the following cases for $B$:
\begin{itemize}
\item $B=p$ is atomic: obvious by definition.
\item $B=\invert p$: since $\invert p\in \Delta$, by $\neg$-cosaturation we have
$p\nin\Delta$. Therefore, by definition, $\kcal_\Gamma,\delta\models p$.
Hence $\kcal_\Gamma,\delta\nmodels \invert p$.
\item $B=\bigwedge \Theta$: by $\wedge$-cosaturation, there is some $C\in\Theta\cap\Delta$.
	Therefore, by the induction hypothesis, $\kcal_\Gamma,\delta\nmodels C$.
	Thus $\kcal_\Gamma,\delta\nmodels \bigwedge\Theta$, as desired.
\item $B=\bigvee\Theta$: by $\vee$-cosaturation, we have $\Theta\subseteq \Delta$.
	Therefore, by the induction hypothesis we get
	$\kcal_\Gamma,\delta\nmodels C$ for
	every $C\in\Theta$.
	Thus $\kcal_\Gamma,\delta\nmodels \bigvee\Theta$, as desired.
\item $B=\Box C$: by $\Box$-cosaturation,
there is some $[\Theta]\in\Delta$ such that
$C\in\Theta$. Let $\theta$ be the occurrence of the mentioned $\Theta$.
Therefore, $\delta\R\theta$.
By the induction hypothesis,
$\kcal_\Gamma,\theta\nmodels C$,
and thus $\kcal_\Gamma,\delta\nmodels \Box C$.
\item $B=\lozenge C$:
let  $\theta\sqsupset \delta$ be an occurrence of some $\Theta$ in $\Gamma$.
By w$\lozenge$-cosaturation, we have $C\in\Theta$. Then by the induction hypothesis,
$\kcal_\Gamma,\theta\nmodels C$. Thus, by definition,
we get $\kcal_\Gamma,\delta\nmodels\lozenge C$, as desired.
\end{itemize}
The above argument works identically for the second statement of the lemma with the following
modifications.
First, replace $\kcal_\Gamma$ with $\kcal^t_\Gamma$
and $\R$ with $\Rt$.
We also have the following argument for the case $B=\lozenge C$:
\begin{itemize}
\item $B=\lozenge C$:
let  $\theta\tR \delta$ be an occurrence of some $\Theta$ in $\Gamma$.
Since $\theta$ is a strict sub-occurrence of $\delta$ and $\lozenge C\in\Delta$,
by $\lozenge$-cosaturation we have $C\in\Theta$. Then by the induction hypothesis,
$\kcal^t_\Gamma,\theta\nmodels C$. Thus, by definition,
we get $\kcal^t_\Gamma,\delta\nmodels\lozenge C$, as desired.\qedhere
\end{itemize}
\end{proof}


The \textit{completeness} of a (nested) sequent calculus $\sft$
for a class $\mathfrak{M}$ of Kripke models is defined as follows: given a (nested) sequent $\Gamma$,
if $\kcal\models \Gamma$ for every $\kcal\in\mathfrak{M}$ supporting $\Gamma$, then $\sft\vdash \Gamma$.

\begin{theorem}\label{dgla-completeness}
We have the following completeness results:
\begin{itemize}
\item $\dskk$ is complete for Kripke models.
\item $\dsk$ is complete for transitive Kripke models.
\item $\dgla$ is complete for transitive and conversely well-founded Kripke models.
\end{itemize}
\end{theorem}
\begin{proof}
Let $\sft$ be any of $\dskk$, $\dsk$ or $\dgla$. Also assume that
$\Gamma_0$ is a nested sequent such that $\sft\nvdash  \Gamma_0$.
By induction on $n$,
we construct a partial $\sft$ proof-tree $T_n$ as follows. The construction is such that
each $T_n$ is a tree of finite depth. Furthermore, $T_{n+1}$ is an extension of $T_n$ obtained by adding
some nodes above the leaves of $T_n$. Also, each
$T_n$ will be well-founded.

For a given nested sequent $\Gamma$ appearing as a label of a leaf in some
$T_n$, let 
$$\gsub{\Delta\ugam_0},\gsub{\Delta\ugam_1},\ldots, \gsub{\Delta\ugam_m}$$
be an enumeration of all occurrences in $\Gamma$.
We make this enumeration in such a way
that for any $\Gamma'$ which appears above $\Gamma$ in later steps of the construction,
the first $m$ occurrences remain the same as for $\Gamma$ (see \cref{sec-limit}).

Also, let $B_0,B_1,\ldots$ be an enumeration of all formulas in $\sub {\Gamma_0}$ such that
every formula occurs infinitely often.
Let $T_0$ be the tree with a single node labelled by the nested sequent $\Gamma_0$.
Also, assume that $n+1=\langle i,j,k\rangle$, i.e.~$n+1$ is the code of the triple $(i,j,k)$ of
numbers. We assume that this coding satisfies $i<\langle i,j,k\rangle $.
We then define $T_{n+1}$ by adding to $T_n$ some nodes on top of
its leaves, as described in what follows.
For each non-trivial\footnote{This means that for every $\Delta$ occurring in $\Gamma$  we have 
$\top\nin\Delta$ and
  for any atomic $p$, either $p\nin\Delta$ or $\invert p\nin\Delta$.} $\Gamma$
appearing as the label of a leaf in $T_n$,
with $\gsub{\Delta_j}$
(notice that this $j$ is the second component of $n+1$) being the   occurrence
of $\Delta_j$ in $\Gamma$,
we add corresponding new node(s) on top of $\Gamma$ according to the following cases.
\begin{itemize}
\item $B_i=\bigvee\Delta\in\Delta_j$ and $B_k\in \Delta$:
add $\gsub{\Delta\ugam_j,B_k}$ on top of $\Gamma$ with its edge labelled by $\vee$.
Notice that $k$ here is the third component of $n+1$.
\item $B_i=\bigwedge\Delta\in\Delta_j$:
for every $B\in\Delta$, add
$\gsub{B,\Delta\ugam_j}$ on top of $\Gamma$ with the label $\wedge$ on the edges.
This in particular means that if  $\Delta=\emptyset$, we do not add anything.
\item $B_i=\Box B\in\Delta_j$:
add $\gsub{[B],\Delta\ugam_j}$ on top of $\Gamma$ with the label $\Box$.
\item $B_i=\lozenge B$ and
  $\gsub{\Delta_j}$ is the immediate sub-occurrence of some other occurrence
$\theta$ of $\Theta$ in $\Gamma$
such that $B_i\in\Theta$:
notice that this implies $[\Delta\ugam_j]\in\Theta$.
Then in the case $\sft=\dskk$,
add $\gsub{\Delta\ugam_j,B}$ on top of $\Gamma$ with the label $\lozenge$.
In the other cases,
add $\gsub{\Delta\ugam_j,B,\lozenge B}$ on top of $\Gamma$ with the label $\lozenge$.
\item Otherwise: add nothing.
\end{itemize}
Then let $T^*:=\bigcup_{n=0}^\infty T_n$.
Given that $\Gamma_0$ is not $\sft$-derivable,
$T^*$ cannot be an $\sft$-proof.
By its definition, it is obvious that $T^*$
is a partial ($\dskk$) proof-tree.
Since $T^*$ is not a  $\sft$-proof, one of the following two cases occurs:
\begin{itemize}
    \item There is some leaf node $\Gamma$ in $T^*$ which is not the conclusion of any of the inference rules.
    This means that $\Gamma$ is cosaturated (in the case of $\sft=\dskk$, weakly cosaturated), and hence by \Cref{saturation-dgla} we have the desired countermodel.
    \item There is an infinite branch  $\Gamma_0,\Gamma_1,\ldots$ in $T^*$ such that in the case $\sft=\dgla$, this branch does not have infinitely many instances of $\Box$-rule nested to each other.
    So for the rest of this proof, assume that such an infinite branch exists.
\end{itemize}
For every $i$, let $S_i$ be a syntax-tree of $\Gamma_i$ as described in \cref{sec-limit}.
 Therefore, for every $i\in\omega$, we have $S_{i+1}\supset S_i$.
Then define $S^*:=\bigcup_{i\in\omega}S_i$.
\\[2mm]
\textit{Claim 1}.
We have the following properties for $S^*$:
\begin{itemize}
    \item[i.]  $S^* $  is a non-well-founded nested tree.
Let $\Gamma^*$ denote the non-well-founded nested sequent corresponding to $S^*$.
    \item[ii.] In the case $\sft=\dgla$,
this branch of the proof includes only finitely many nested applications of
the $\Box$-rule, and hence in this case $\Gamma^*$ is a general nested
sequent. However, in the two other cases, $\Gamma^*$ is just a non-well-founded
nested sequent.
\item[iii.] For any occurrence $\delta$  of some $\Delta^*$ in $\Gamma^*$ and any
$B\in\Delta^*\cap\langba$, there is some $n$ such that $\delta$ is also an occurrence
of some $\Delta_n$ in $\Gamma_n$ with $B\in\Delta_n$. Furthermore, $\delta$
remains an occurrence of some $\Delta_m$ in $\Gamma_m$  with $B\in\Delta_m$, for any $m>n$.
\end{itemize}
\vspace{1mm}
\textit{Proof of Claim 1.}
 The proofs of all the above-mentioned properties are straightforward and left to the reader.
\\[2mm]
\textit{Claim 2.} If $\sft=\dskk$, then $\Gamma^*$ is weakly cosaturated.
Otherwise, it is cosaturated.
\\[2mm]
\textit{Proof of Claim 2.} We need to show that $\Gamma^*$ satisfies all the conditions for
being (weakly) cosaturated:
\begin{itemize}
\item
\textit{$\neg$-cosaturated}:
let $p,\invert p\in\Delta^*$ for some occurrence $\delta$ of $\Delta^*$ in
$\Gamma^*$.
Hence, by Claim 1, there is some $n$ such that $\Gamma_n$ is trivial.
This means that by the definition of $T^*$, there should be no node above $\Gamma_n$ in the proof-tree $T^*$,
contradicting the infiniteness of
the sequence $\{\Gamma_i\}_{i\in\omega}$.
\item
\textit{$\vee$-cosaturated}:
let $\bigvee\Theta\in\Delta^*$ for some occurrence $\delta$ of $\Delta^*$ in $\Gamma^*$ and let
$B\in\Theta$.
Hence, by Claim 1, there is some $n$ such that $\bigvee\Theta$
belongs to the same occurrence $\delta$ of some  $\Delta$ in
$\Gamma_n$.
Choose $i>n$ such that $B_i=\bigvee\Theta$,
let $j$ be the index of the occurrence $\delta$ of $\Delta$ in $\Gamma_n$, and let $k$ be such that $B_k=B$.
Then for  $m:=\langle i,j,k\rangle $ we have $m>n$.
This implies\footnote{Here we are using the extra assumption
that the indices of the occurrences of $\Gamma$ remain unchanged in later steps.}
that in $T_m$ we must add $B$ to the occurrence $\delta$ of
$\Delta'$ in $\Gamma_m$. Thus $B\in\Delta' $, where $\delta$ is
the occurrence of $\Delta'$ in $\Gamma_m$. Therefore, by Claim 1,
$B\in \Delta^*$ (notice that $B\in  \langba$), as desired.
\item
\textit{$\wedge$-cosaturated}: let $\bigwedge\Theta\in\Delta^*$
for some occurrence $\delta$ of $\Delta^*$ in $\Gamma^*$.
Hence, by Claim 1, there is some $n$ such that $\bigwedge\Theta$
belongs to the same occurrence $\delta$ of  some $\Delta $ in
$\Gamma_n$.
Choose $i>n$ such that $B_i=\bigwedge\Theta$,
and let $j$ be the index of the occurrence $\delta$ of $\Delta$ in $\Gamma_n$,
i.e.~$\Delta=\Delta\ugam_j$. Then for  $m:=\langle i,j,0\rangle $  we have $m>n$.
This means that in $T_m$, for some $B\in\Theta$,
we must add $B$ to the occurrence $\delta$ of
$\Delta'$ in $\Gamma_m$.
Thus $B\in \Theta\cap\Delta' $, where $\delta$ is
the occurrence of $\Delta'$ in $\Gamma_m$.
Therefore, $B\in\Theta\cap\Delta^*$, as desired.
\item
w\textit{$\lozenge$-cosaturated}: (in the case $\sft=\dskk$)
let $\lozenge B,[\Delta^*]\in\Theta^*$
for some occurrence $\theta$ of $\Theta^*$ in $\Gamma^*$.  Also let $\delta$ be the
occurrence of $[\Delta^*]$ in $\Gamma^*$. This means that $\delta$ is an immediate
 sub-occurrence of $\theta$.
Hence, by Claim 1, there is some $n$ such that $\lozenge B$
belongs to the same occurrence $\theta$ of $\Theta$ in
$\Gamma_n$ and also $\delta$ is an occurrence of some $\Delta$ in
$\Gamma_n$\footnote{Notice that this occurrence might be for some nested
sequent different from $\Delta^*$.} such that $\delta$ is an immediate sub-occurrence of $\theta$.
Choose $i>n$ such that $B_i=\lozenge B$,
and let $j$ be the index of the occurrence $\theta$ of $\Theta$ in $\Gamma_n$.
Then for  $m:=\langle i,j,0\rangle $  we have $m>n$.
Then let $\Delta_m$ and $\Theta_m$ be such that $\delta$ and $\theta$ are the occurrences of
$\Delta_m$ and $\Theta_m$ in $\Gamma_m$, respectively.
This means that
$[\Delta_m]\in \Theta_m$ and
we must have $B\in\Delta_m$.
Therefore, by Claim 1, we get $B\in\Delta^*$, as desired.
\item
\textit{$\lozenge$-cosaturated}:
(in the case $\sft=\dsk,\dgla$)
almost identical to the proof of the previous case.
\item
\textit{$\Box$-cosaturated}:
let $\Box B\in\Delta^*$
for some occurrence $\delta$ of $\Delta^*$ in $\Gamma^*$.
Hence, by Claim 1, there is some $n$ such that $\Box B$
belongs to the same occurrence $\delta$ of $\Delta \subseteq\Delta^*$ in
$\Gamma_n$.
Choose $i>n$ such that $B_i=\Box B$,
and let $j$ be the index of the occurrence $\delta$ of $\Delta$ in $\Gamma_n$,
i.e.~$\Delta=\Delta\ugam_j$. Then for  $m:=\langle i,j,0\rangle $  we have $m>n$.
Also assume that $\delta$ is the occurrence of $\Delta'$ in $\Gamma_m$.
This means that $[B]$ must belong to $\Delta'$. Then Claim 1 implies
 that there is some $[\Theta^*]\in\Delta^*$ such that $B\in\Theta^*$.
 More precisely, the occurrence of this $[\Theta^*]$ in $\Gamma^*$ is the same as the
 occurrence of the added $[B]$ in $\Gamma_m$.
\end{itemize}
Now with the aid of Claim 2 and \Cref{saturation-dgla} we obtain the desired countermodel.
\end{proof}

\subsection{Cut: admissibility and elimination}
The admissibility of cut is a well-known corollary of the soundness and completeness theorems for a cut-free sequent calculus.
Here we also prove the admissibility of cut for all the sequent calculi $\dskk$, $\dsk$ and $\dgla$.
The cut rule in the setting of the Tait-style nested sequent calculus takes the following form:

\begin{prooftree}
    \Ax{$\gsub{\Delta,A}$}
    \Ax{$\gsub{\Delta,\neg A}$}
    \LLa{cut}
    \BI{$\gsub{\Delta}$}
\end{prooftree}
The admissibility of cut for  $\sft$ then means that, if we have both $\sft\vdash \gsub{\Delta,A}$
and $\sft\vdash \gsub{\Delta,\neg A}$, then $\sft\vdash \gsub{\Delta}$.

\begin{theorem}\label{cut}
Cut is admissible in $\dskk$, $\dsk$ and $\dgla$.
\end{theorem}
\begin{proof}
    We reason contrapositively.
    Let $\sft$ be any of the calculi mentioned, and suppose $\sft\nvdash \gsub\Delta$.
    Then by the completeness theorem \ref{dgla-completeness}, there is a Kripke model $\kmodel$
    and $w\in W$ such that $\kcal,w\nmodels \gsub\Delta$ and $\kcal$ has the frame condition corresponding to $\sft$.
    Let $n$ be the rank of the occurrence $\gsub\Delta$.
    Then \Cref{lem-cut} implies that there is some $u\in W$ such that $\reachable w n u$ and $\kcal,u\nmodels\Delta$.
    We have either $\kcal,u\nmodels A$ or $\kcal,u\nmodels \neg A$. Therefore, either $\kcal,u\nmodels(\Delta,A)$
    or $\kcal,u\nmodels(\Delta,\neg A)$.  Then \Cref{lem-cut} implies that either $\kcal,w\nmodels \gsub{\Delta,A}$ or
    $\kcal,w\nmodels\gsub{\Delta,\neg A}$. Thus, by the soundness theorems \ref{dsk-soundness} and \ref{dgla-soundness},
    we get either $\sft\nvdash \gsub{\Delta,A}$ or $\sft\nvdash \gsub{\Delta,\neg A}$, as desired.
\end{proof}

Let $\dskkc$, $\dskc$ and $\dglac$ respectively denote the additions of the cut rule to the
calculi $\dskk$, $\dsk$ and $\dgla$. Cut-elimination for
a sequent calculus $\sft$ means that, even after removing the cut rule from the inference rules of $\sft$,
we can still prove all theorems of $\sft$.
It is obvious from \Cref{cut} that $\dskkc$ and $\dskc$ enjoy cut-elimination. For example,
the reasoning for cut-elimination in $\dskkc$ is by easy induction from top to bottom on
the $\dskkc$ proof-tree. Namely, we start from the leaves of the proof tree and one by one replace all frontier instances of cut with some other proof without cut, given by \Cref{cut}.
Since the proof-trees here are well-founded, such an induction is legitimate.
Nevertheless, this argument does not work for $\dglac$, since the proofs there are not well-founded.

\begin{theorem}\label{cut-elimination}
All of the following sequent calculi enjoy cut-elimination: $\dskkc$, $\dskc$ and $\dglac$.
\end{theorem}
\begin{proof}
    We only prove cut-elimination for $\dglac$. Assume that $\dglac\vdash \Gamma$ for some
    nested sequent $\Gamma$. It is straightforward to observe that the soundness theorem
    \ref{dgla-soundness} can be extended to the soundness of $\dglac$. Therefore, $\Gamma$ is valid in
    all conversely well-founded transitive Kripke models. Then by completeness of $\dgla$ (\Cref{dgla-completeness}) we get $\dgla\vdash\Gamma$, as desired.
\end{proof}




\section{Formalized provability for set theories in infinitary languages}


The main goal of this section is to provide a framework for the study of what strong enough
infinitary theories can prove about provability in themselves. To accomplish this, we develop
an infinitary counterpart of the study of arithmetized provability predicates for arithmetical theories.
In particular, we use our framework to provide an infinitary provability semantics for the language
of infinitary modal logic.

Kripke-Platek set theory is a rather weak set theory: its axioms are Extensionality, Pair,
Union, Foundation (for arbitrary formulas), $\Delta_0$-Separation and $\Delta_0$-Collection.
It is, in a sense, the minimal theory that allows the development of the basics of recursion theory.
In spite of these limitations, it allows one to carry out a good amount of set theory, and has long been
a focus of study for the area of logic at the intersection of set theory, model theory and
recursion theory. A transitive set or class that is a model of $\kp$ is called \emph{admissible}.
Examples of admissible sets are given by $\mathbb{HF}$, the class of hereditarily finite sets,
and $L_{\omega_1^{ck}}$, where $\omega_1^{ck}$ is the smallest non-recursive ordinal.
We refer to \cite{barwise} for a very nice introduction to several aspects of Kripke-Platek set theory.

A very important feature of $\kp$ that we are going to use repeatedly is that it supports a good theory
of inductive definitions: namely, Gandy's theorem states that $\kp$ proves that the least fixed point
of a positive $\Sigma_1$ operator is definable by a $\Sigma_1$ formula. In the study of formalized
provability, many natural notions such as languages and provability relations are instances of
inductive definitions. Therefore, Gandy's theorem allows us to define most of the objects we
are interested in in a natural way.

Since \cite{barwise1969infinitary} it has been known that the infinitary first-order logic restricted
to admissible sets\footnote{By this we mean that infinite conjunctions and disjunctions are
allowed over sets of formulas which already belong to a fixed admissible set.}
is rather well-behaved. Thus, for our purpose of studying formalized provability
in the infinitary setting, the natural approach is, for each admissible set $A$, to consider
the $A$-fragment of the infinitary first-order language in the (set-theoretic) signature of $\kp$.
However, one relevant aspect in which these languages differ from the finitary language
of $\mathsf{PA}$ is that they possess no counterpart of the numerals (closed terms representing all
individual natural numbers). To resolve this deficiency we further expand the $A$-fragment of
the infinitary set-theoretic language by the family of constants $\{\Godelnum{a} \mid a \in A\}$,
and denote the resulting language by $\langl[A]$. We thus define a language in which we
have terms (in fact constants) referring to all individual formulas of the language itself,
mimicking the standard setup for the study of arithmetized provability, where arithmetical formulas
are representable within the arithmetical language as numerals for their G\"odel numbers.

In the study of formalized provability for arithmetical theories, $\Sigma_1$
arithmetical formulas play a crucial role since the formalized provability predicate is
itself $\Sigma_1$ and strong enough arithmetical theories prove all true $\Sigma_1$-sentences
(are $\Sigma_1$-complete). Thus we define the class of $\Sigma[A]$ formulas, 
which is a natural $\langl[A]$-counterpart of arithmetical $\Sigma_1$-formulas. In the case of
arithmetical theories, even Robinson's arithmetic $\mathsf{Q}$ ($\mathsf{PA}$ with induction dropped)
is $\Sigma_1$-complete. However, in our setting, because $A$ can be any admissible set, we have to
add axioms specific to the admissible set to recover the analogous property of $\Sigma[A]$-completeness
for our infinitary theories. This naturally leads to the following minimalistic theory
$\mathsf{R}[A]$, whose axioms are essentially the specifications of the contents of all individual sets $a\in A$.
The main point about $\mathsf{R}[A]$ is that it precisely axiomatizes the set of all formulas from the class $\Sigma[A]$ that are true in $A$.
Thus, for each admissible set $A$, we choose the theory $\kp[A]=\kp+\mathsf{R}[A]$ as our base theory
for the study of formalized provability.

Within this infinitary theory, we can define, using Gandy's theorem, what it means to be provable,
and give a definition of truth for $\Sigma[A]$ formulas in the infinitary language $\langl[A]$:
this is done in \Cref{def:deriv-inf-lang} and in \Cref{def:truth-lsigma}. All this is done in a
way analogous to what usually happens for the standard case of $\mathsf{PA}$. We notice in
\Cref{rem:proof-system}, however, that the relationship between the notion of derivability
and the notion of proof is more complicated in our setting than it is in arithmetic.

In this setting, we can prove the
Hilbert-Bernays-L\"ob derivability conditions and the Gödel-Löb axiom (\Cref{theo:deriv-ax} and
\Cref{cor:lobs-ax}) in a standard way: we follow the approach presented in \cite{logic-prov-boolos}
to achieve this. Moreover, we notice in \Cref{lem:barcan} that Barcan's formula holds as well.

Again using Gandy's theorem, we can introduce our infinitary version of the modal language
$\langl[A]^{\Box}$, a Hilbert-style calculus $\hgl[A]$ for this language, and the concept
of an interpretation $^*$ of a formula of this language into $\langl[A]$, translating $\Box$
as the provability predicate of $\kp[A]$. $\hgl[A]$ essentially corresponds to $\HGL {\omega_1}$,
introduced in \Cref{sec:hgl-omega}, although there are some subtle points in their relationship.

We conclude the section with \Cref{theo:low-bound-kp}, which proves that if a modal sentence $B$
is derivable in $\hgl[A]$, then for every interpretation $(\,)^*$, $\kp[A]$ proves $B^*$.

\bigbreak



We now start to present more formally the ideas given in the previous paragraphs.
The metatheory that we use for this section diverges from the style of presentation familiar from the book of Barwise.
Instead of fixing an admissible set, we have a fixed admissible class $A$.
We now define our meta-theory in more detail. The language of the meta-theory expands the language of set theory by an extra unary predicate $x\in A$ (to be read as ``$x$ is an element of the class $A$''). We can then naturally talk about $A$-bounded quantifiers $\exists x\in A\, \varphi$ and $\forall x\in A\, \varphi$, which should be treated as shorthand for $\exists x (x\in A \land \varphi)$ and $\forall x(x\in A\to \varphi)$, respectively. And we can talk about the relativization $\varphi^A$ of a formula $\varphi$ to the class $A$, obtained by replacing each unbounded quantifier $\exists x \,\varphi$ or $\forall x \,\varphi$ with the corresponding $A$-bounded quantifier $\exists x\in A\,\varphi$ or $\forall x\in A\, \varphi$, respectively.
The metatheory is the extension of $\kp$ by:
\begin{enumerate}
\item $\forall x\in A\forall y\in x(y\in A)$ ($A$ is transitive);
\item $\varphi^A$, for each axiom $\varphi$ of $\kp$;
\item extension of the scheme of foundation to all formulas of the expanded language.
\end{enumerate}
We denote this theory by $\kp + \mathsf{Tr}(A) + \kp^A$ and informally call it ``$\kp$ plus $A$ is an admissible class''.

The idea behind this approach is that it allows us to treat the external and the internal perspectives on the admissible set at the
same time. Namely, on the one hand, one way in which we will be using the results obtained here is by interpreting this theory within $\kp$ (or its extensions) by setting $A$ to be just $V$, which corresponds to ``internal reasoning'' within an admissible set.
Another option is to instead interpret this theory in $\kp$ by setting $A$ to be some admissible set $\adma$.
Under the second interpretation, the results that we prove will be ``external'' facts about admissible sets that
are established in $\kp$ as a meta-theory. Thus this approach allows us to unify the two perspectives.

For the purposes of this section, the main perspective that should be kept in mind is the ``internal'' perspective,
where  $A$ is interpreted as $V$. However, in the next section,
and in particular in \Cref{theo:compl-di-syst}, we do use the ``external'' perspective when linking the deep inference system $\dgla$ with the provability interpretation of the modalities.



In what follows, we denote by $\Sigma^A$ the class of formulas $\varphi$ that are of the form $\psi^A$ for some $\Sigma$-formula $\psi$.  When reasoning in $\kp + \mathsf{Tr}(A) + \kp^A$ and asserting that something is definable by a $\Sigma^A$-formula, we furthermore assume that the formula can have parameters, but those parameters have to be from the class $A$.

The following is a crucial technical result that we will repeatedly use in this section to allow us to formally introduce various notions via inductive definitions.
\begin{theorem}[Gandy's theorem for classes]\label{thm:ind-def}
  Let $\varphi(\vec x, S)$ be a
  $\Sigma^A$-formula with an extra fresh predicate letter $S$ of the same arity as the list of variables $\vec{x}$ such that $S$ occurs only positively.
  Then, there is a $\Sigma^A$-formula $\mathsf{Fix}_\varphi(\vec x)$ such that
  \[\kp + \mathsf{Tr}(A) + \kp^A\vdash \mathsf{Fix}_\varphi(\vec x) \biimp \forall \vec{x}\in A (\varphi(\vec x,\lambda \vec{x}.\mathsf{Fix}_\varphi(\vec x))),\]
  and furthermore $\kp + \mathsf{Tr}(A) + \kp^A$ verifies that $\mathsf{Fix}_\varphi$ describes the minimal upper pre-fixed point of $\varphi$; more formally the latter condition means that
  \[\kp + \mathsf{Tr}(A) + \kp^A\vdash \forall \vec{x}\in A (\varphi(\vec{x},\lambda\vec{x}. \theta(\vec{x}))\to \theta(\vec{x}))\to \forall \vec{x}\in A(\mathsf{Fix}_\varphi(\vec{x})\to \theta(\vec{x})),\]
  for any formula $\theta$ of the language of $\kp + \mathsf{Tr}(A) + \kp^A$.
\end{theorem}
We sketch a proof of this Theorem in \cref{sec:proof-gandy-classes}.

We now illustrate how Gandy's theorem is used with a toy example of an inductive definition; later we will use it freely in an analogous manner for more complex inductive definitions. Suppose we want to define the relation $x\in^* y$ on $A$ meaning that a set $x$ is an element of the transitive closure of $y$. This is naturally the smallest binary relation  $\in^*$ such that
\begin{enumerate}
    \item $x\in y$ $\Rightarrow$ $x\in^* y$;
    \item $\exists y'\in y(x\in^* y')$ $\Rightarrow$ $x\in^* y$.
\end{enumerate}
Note that here we have a list of clauses (two for this toy example), each of which says that if some condition $\chi_i(x,y,\in^*)$
holds, then $x\in^* y$. In our case, $\chi_0(x,y,S)$
is $x\in y$ and $\chi_1(x,y,S)$ is $\exists y'\in y(S(x,y'))$. An important property of these clauses is that the occurrences of the binary predicate $\in^*$ in $\chi_i(x,y,\in^*)$ are all positive.  We now define the formula $\varphi(x,y,S)$ to be
$\chi_0(x,y,S)\lor \chi_1(x,y,S)$; in the general case we take the disjunction over all clauses of a given
inductive definition. As long as all the $\chi_i(x,y,S)$ are $\Sigma$-formulas with all occurrences of $S$ positive (which, of course, is true in our particular example), then so is $\varphi(x,y,S)$.
We then use Gandy's theorem to construct the corresponding fixed point, which we denote by $\in^*$.
The minimality condition from Gandy's theorem is the condition that $\in^*$ is indeed the smallest
binary relation satisfying the two implications above.

We note that the construction in the previous paragraph naturally gives rise to a monotone operator $\Gamma_{\in^*}(S)$ on $\Sigma^A$-definable binary relations. Namely, for a binary relation $R$ on $A$ given by $\psi(x,y)$, the value $\Gamma_{\in^*}(R)$ is exactly the binary relation definable by $\varphi(x,y,\lambda x,y.\psi)$. Of course, $\in^*$ is the least fixed point of this operator. And the minimality principle from Gandy's theorem can be phrased as the property that for any other $\Sigma^A$ binary relation $R$, if $R\supseteq \Gamma_{\in^*}(R)$, then $R\supseteq (\in^*)$.

We can use this machinery of inductive definitions to naturally define infinitary first-order languages.
Namely, given a $\Delta^A$-definable first-order signature, using a straightforward coding we define the $A$-fragment of the
infinitary language of this signature as the least class of formulas containing all atomic formulas and closed under
the formation of formulas using the connectives $\lnot,\to$, the quantifiers $\forall$ and $\exists$, and the formation of infinitary conjunctions
$\bigwedge \Phi$ and disjunctions $\bigvee \Phi$, as long as the sets of formulas $\Phi$ lie in $A$.
The specific case that we are interested in in the present paper is the language with equality $=$, the binary
membership predicate $\in$, and the constants $\Godelnum{a}$ for each $a\in A$.
Since this definition can be carried out via Gandy's Theorem, $\langl[A]$ is a $\Sigma^A$-class.
We denote this infinitary language by $\langl[A]$. We give a formal
definition of $\langl[A]$ in \Cref{def:langl}.

Next, we define what it means for a formula $\varphi$ of $\langl[A]$ to be provable from a certain theory.
This is analogous to an approach to the definition of theoremhood suggested in \cite{barwise}, along with
several other places, but we prefer to stay closer to the Hilbert calculus $\HCLA$ introduced in \cref{sec-HCLA}.


\begin{definition}\label{def:deriv-inf-lang}
  An $A$-\emph{theory} $\ttt$ is a $\Sigma^A$-subclass of $\langl[A]$-sentences. Elements of $\ttt$ are called \emph{axioms} of $\ttt$.

  For $A$-theories we define the relation $\ttt\vdash^A \varphi$
  on $\langl[A]$ as the smallest relation such that:
  \begin{itemize}
  \item $\ttt\vdash^A \varphi$ holds for every $\varphi\in \ttt$;
  \item $\ttt\vdash^A \varphi$ holds when $\varphi$ is an instance of a substitution of $A,B,C$ with formulas $\theta,\psi,\eta$ in $\langl[A]$ in the axioms  $1$ to $6$ of $\HCLA$;
  \item if $\ttt\vdash^A \varphi$ and $\ttt\vdash^A \varphi\imp\psi$, then $\ttt\vdash^A \psi$
  (modus ponens);
  \item for every set of $\langl[A]$-formulas $\Phi\in A$ and $\varphi\in \Phi$, we have the axiom $\bigwedge\Phi\to \varphi$ (conjunction axioms);
  \item if $\ttt\vdash^A \varphi\imp\psi$ for every $\psi$ in a set of $\langl[A]$ formulas
  $\Psi\in A$, then $\ttt\vdash^A \varphi\imp\bigwedge \Psi$ (conjunction rule);
  \item for every set of $\langl[A]$-formulas $\Phi\in A$ and $\varphi\in \Phi$, we have the axiom $\varphi\to\bigvee \Phi$ (disjunction axioms);
  \item if $\ttt\vdash^A \varphi\imp\psi$ for every $\varphi$ in a set of $\langl[A]$ formulas
  $\Phi\in A$, then $\ttt\vdash^A \bigvee \Phi \imp \psi$ (disjunction rule);
  \item $\forall x \varphi(x)\to \varphi(x)$, for each $\varphi(x)\in\langl[A]$ and $\langl[A]$-term $t$;
  \item if $v\in\Var$ is not a free variable of $\varphi$ and $\ttt\vdash^A \varphi\imp\psi$,
  then $\ttt\vdash^A \varphi\imp\forall v \psi$ ($\forall$-Bernays rule);
    \item $\varphi(t)\to\exists x \varphi(x)$, for each $\varphi(x)\in\langl[A]$ and a $\langl[A]$-term $t$;
  \item if $v\in\Var$ is not a free variable of $\psi$ and $\ttt\vdash^A \varphi\imp\psi$,
  then $\ttt\vdash^A \exists v \varphi\imp \psi$ ($\exists$-Bernays rule);
  \item $\ttt\vdash^A \varphi$, for all of the axioms of equality: \begin{itemize} \item $\ttt\vdash^A x=y\to (\varphi(x) \to \varphi(y))$, for each $\varphi(x)\in \langl[A]$,
  \item $\ttt\vdash^A x=x$,
  \item $\ttt\vdash^A x=y \to (y=z \to x=z)$, \item $\ttt\vdash^A x=y \to y=x$.\end{itemize}
  \end{itemize}
  Let $\proov(x,\varphi)$ be a $\Delta_0$ formula
  defining $\ttt\vdash^A\varphi$, i.e.\ the formula such that $\ttt\vdash^A\varphi$ if and only
  if $\exists x\in A(\proov (x,\varphi))$. Such a formula exists since in $\kp$ every $\Sigma$-formula $\psi$ can be equivalently transformed to a formula of the form $\exists x\psi'(x)$, where $\psi'$ is $\Delta_0$.  In our case we do this transformation within the admissible class $A$.
\end{definition}

\begin{remark}
    In terms of Barwise's notation, the previous definition in essence says that $\ttt\vdash^A\varphi$
    holds if $\varphi$ is contained in every validity property that contains $\ttt$, rather than saying that $\varphi$ has an infinitary proof. The two notions are, of course, equivalent.
\end{remark}

\begin{remark}\label{rem:proof-system}
    Analogously to what happens in the case of admissible sets, we should be careful
    about the meaning of the formula $\proov(x,\varphi)$: namely, we should not view $x$
    as giving a proof tree of $\varphi$ in the standard sense of the term, due to the fact that,
    in general, $A$ will not be a model of the axiom of choice.
    To give an intuitive idea of the problem, suppose that we want to deduce that, for a set
    $\Phi$ of formulas, if every $\varphi\in \Phi$ has a proof in $A$, then so does $\bigwedge \Phi$.
    If we were to use a ``standard'' notion of proof, then a proof $p_{\bigwedge\Phi}$ of
    $\bigwedge\Phi$ would be something like a function mapping every $\varphi\in\Phi$ to a
    proof $p_\varphi$: the problem is that, in the absence of the axiom of choice,
    we cannot pick \emph{one} proof of $\varphi$. What we can do in $\kp$ is use strong
    $\Sigma$ replacement to produce a function $f$ whose domain is $\Phi$ and such that $f(\varphi)\neq\emptyset$ for every $\varphi$ in the domain, and every $p\in f(\varphi)$ is a
    proof of $\varphi$. Hence, we adapt our notion of proof to allow for this kind of derivation,
    where ``multirule'' derivations of this form are permitted. We remark that this is exactly
    the approach followed in \cite{barwise}: we refer in particular to Chapter III.5
    there for a more thorough discussion of the topic.

\end{remark}

A class of formulas playing a role analogous to the role of $\Sigma_1$-formulas for arithmetical theories is the class of $\LSigma[A]$ formulas. These are the obvious analogue of $\Sigma$ formulas for the infinitary language $\langl[A]$. Namely, the class of $\LSigma[A]$ formulas is the smallest class of $\langl[A]$ formulas containing all the literals (i.e., all atomic formulas and their negations) and closed under infinitary conjunction, infinitary disjunction, bounded universal quantification, and existential quantification. It is easy to formalize this definition using Gandy's theorem, which gives that the class of $\LSigma[A]$ formulas is a $\Sigma^A$ class.


Among their most important properties is the fact
that $\LSigma[A]$ formulas admit a simple (i.e., $\Sigma^A$) definition of truth $\models^A_\Sigma$. We formally define it below.
As usual, the definition is given implicitly using Gandy's Theorem.
\begin{definition}\label{def:truth-lsigma}
  We let $\models^A_\Sigma\varphi$ be the least relation on the $\Sigma^A$-class of the $\LSigma[A]$-sentences such that:
  \begin{itemize}
  \item if $\varphi$ is $\Godelnum{a}\in\Godelnum{b}$ and $a\in b$, then $\models^A_\Sigma\varphi$, and
  \item if $\varphi$ is $\Godelnum{a}=\Godelnum{a}$, then $\models^A_\Sigma\varphi$, and
  \item if $\varphi$ is $\neg\psi$ and $\psi$ is $\Godelnum{a}\in \Godelnum{b}$ for $a\not\in b$, then $\models^A_\Sigma\varphi$, and
  \item if $\varphi$ is $\neg \psi$ and $\psi $ is $ \Godelnum{a}=\Godelnum{b}$ for $a\neq b$, then $\models^A_\Sigma\varphi$, and
  \item if $\varphi$ is $\bigwedge  \Psi$, for some set of formulas $\Psi$, and for every $\psi\in \Psi$ we have that $\models^A_\Sigma\psi$, then $\models^A_\Sigma\varphi$, and
  \item if $\varphi$ is $\bigvee  \Psi$, for some set of formulas $\Psi$, and for at least one $\psi\in \Psi$ we have that $\models^A_\Sigma\psi$, then $\models^A_\Sigma\varphi$, and
  \item if $\varphi$ is $\forall x\in \Godelnum{a} \psi(x)$ and for every $b\in a$ we have $\models^A_\Sigma \psi(\Godelnum{b})$, then $\models^A_\Sigma \varphi$, and
  \item if $\varphi$ is $\exists x \psi(x) $ and for some $a\in A$ we have that $\models^A_\Sigma \psi(\Godelnum{a})$, then $\models^A_\Sigma \varphi$.
  \end{itemize}
  We denote the corresponding positive operator $\Gamma_{\models^A_\Sigma}(X)$.
\end{definition} 

We want to see that, for sentences of $\LSigma[A]$, there is a correspondence between
truth and provability from a theory $\ttt$, provided that $\ttt$ is strong enough.
\begin{definition}\label{def:r-a}
  We let $\mathsf{R}[A]$ be an $A$-theory with the following axioms:
\begin{enumerate}
    \item \label{R_A_ax1}$\Godelnum{a} \in \Godelnum{b}$, for $a,b\in A$ such that $a\in b$;
    \item \label{R_A_ax2} $\neg(\Godelnum{a}\in \Godelnum{b})$, for $a,b\in A$ such that $a\not\in b$;
    \item \label{R_A_ax3} $\forall x (x\in \Godelnum{a} \imp \bigvee_{b\in a}(x=\Godelnum{b}) )$, for $a\in A$.
\end{enumerate}
\end{definition}

\begin{remark}
    Our theory $\mathsf{R}[A]$ is a direct analogue of Robinson's arithmetic $\mathsf{R}_0$. The theory $\mathsf{R}_0$
    is axiomatized by numerical instances of certain elementary
    facts of arithmetic and has exactly the same theorems as the arithmetical theory axiomatized by all true $\Sigma_1$-sentences.
    It is a slight alteration of Robinson's arithmetic $\mathsf{R}$ from \cite{tarski1953undecidable},
    and can be found in more modern presentations of the subject, for instance, \cite{beklemishev2010godel}.
\end{remark}

\begin{lemma}\label{lem:truth-imp-prov}
  Let $\varphi$ be a $\LSigma[A]$ sentence, and let $\ttt$ be an $A$-theory extending $\mathsf{R}[A]$. Then, if $\models^A_\Sigma\varphi$ holds,
  so does $\ttt\vdash^A \varphi$.
\end{lemma}
\begin{proof}


    We proceed by induction on the complexity of the formula $\varphi$.

    If $\varphi$ is an atomic formula, then the claim basically corresponds to the fact that $\mathsf{R}[A]$ proves every true closed literal of the language $\langl[A]$ (sentences of the form $\Godelnum{a}\in \Godelnum{b}$, $\lnot \Godelnum{a}\in \Godelnum{b}$, $\Godelnum{a} = \Godelnum{b}$, or $\lnot \Godelnum{a} = \Godelnum{b}$).
    For this fact, the cases of true membership relations follow from the axioms \ref{R_A_ax1}.\ of $\mathsf{R}[A]$, and the cases of true negated membership relations follow from the axioms \ref{R_A_ax2}.\ of $\mathsf{R}[A]$. The only cases of true equalities are $\Godelnum{a}=\Godelnum{a}$, which constitute a first-order provable property of the equality predicate. Finally, to prove sentences of the form $\lnot \Godelnum{a} = \Godelnum{b}$ for $a\ne b$, we reason as follows. Without loss of generality we assume that there is $c\in a$ such that $c\notin b$ (the reasoning when there is $c\in b$ such that $c\notin a$ is analogous). Then we note that $\mathsf{R}[A]\vdash^A \Godelnum{c} \in \Godelnum{a}\land \lnot \Godelnum{c}\in \Godelnum{b}$, and via first-order reasoning with equality we conclude that $\mathsf{R}[A]\vdash^A \lnot \Godelnum{a} = \Godelnum{b}$.

    The cases of conjunction and disjunction are easy.
    Suppose $\varphi$ is $\bigwedge\Psi$; then from the assumption $\models^A_\Sigma\varphi$
    we have that $\models^A_\Sigma\psi$ holds for every $\psi\in\Psi$, and hence so does $\ttt\vdash^A\psi$.
    An application of the rule of conjunction gives the desired result (technically,
    we should first conclude that $\ttt\vdash^A \top \imp \psi$, where $\top$, as before, is $\bigwedge \emptyset$).

    The case of disjunction is analogous to the case of conjunction.


    We now deal with the case of $\varphi$ being $\exists x\psi$. Let $a\in A$ be such that
    $\models^A_\Sigma \psi(\Godelnum{a})$; then by the inductive assumption we have that $\ttt\vdash^A  \psi(\Godelnum{a})$.
    Since $\psi(\Godelnum{a})\imp \exists x\psi(x)$ is one of the axioms of $\ttt$, by modus ponens
    we conclude that $\ttt\vdash^A \varphi$.

    Finally, suppose that $\varphi$ is $\forall x\in \Godelnum{a}\,\psi(x)$. Then $\models^A_\Sigma \psi(\Godelnum{b})$
    for every $b\in a$, which gives $\ttt\vdash^A \psi(\Godelnum{b})$ for every $b\in a$. 
    Since $\{\psi(\Godelnum{b}): b\in a\}$ is a set of formulas,
    we can conclude that $\ttt\vdash^A \bigwedge_{b\in a} \psi(\Godelnum{b})$.
    From this, using infinitary first-order reasoning,
    we can conclude that \[\ttt\vdash^A \forall x (\bigvee_{b\in a}(x=\Godelnum{b})\imp \psi(x)).\]
    On the other hand, by the assumption that $\ttt$ extends $\mathsf{R}[A]$, we have that \[\ttt\vdash^A \forall x(x\in \Godelnum{a}\imp \bigvee_{b\in a} (x=\Godelnum{b})).\]
    We can conclude that
    $\ttt\vdash^A \forall x(x\in \Godelnum{a}\imp \psi(x))$ by the axioms concerning the universal
    quantifier and the classical derivations of the transitivity of implication in Hilbert-style systems,
    followed by a final application of the generalization rule.
\end{proof}

\begin{remark}
    Since $\mathsf{R}[A]$ proves all true $\Sigma[A]$-sentences and all its axioms are $\Sigma[A]$-sentences,  $\mathsf{R}[A]$ is deductively equivalent to the theory axiomatized by precisely all true $\Sigma[A]$-sentences. 
\end{remark}

Next, we want to study the behavior of the provability relation within a theory.
To do this, we need this theory to be somewhat strong.

\begin{definition}\label{def:kpa-ta}
    We let $\kp[A]$ be the $A$-theory whose axioms are the axioms of $\mathsf{R}[A]$
    and the axioms of $\kp$.

    In the following, we let $\ttt$ be a $\Sigma^A$ theory of $\langl[A]$ extending $\kp[A]$.
\end{definition}


The main point of choosing $\kp[A]$ as the ``minimal theory'' we consider is that it effortlessly
allows us to formalize, within it, the constructions of the language, of the relation of provability,
of the definition of $\LSigma[A]$ formulas, and of truth.
This is due to the fact that $\kp[A]$ is an extension of $\kp$.
We view $\kp$ as $\kp + \mathsf{Tr}(A) + \kp^A$ with $A$ interpreted as $V$,
and thus we can still apply Gandy's theorem to $\kp[A]$: we notice that no
infinitary modification of Gandy's theorem is needed, since all of the formulas defining
the classes listed above are finitary. More specifically, we can do the following:
\begin{itemize}
    \item 
    Notice that we can define a coding of $\langl[A]$ in a very natural way:
    since every formula $\varphi$ of $\langl[A]$ is an element of $A$, we can code it by the
    constant $\Godelnum{\varphi}$, and similarly for the elements of the signature (i.e., connectives, variables and constants). 
    It is clear that the class of
    codes can be defined by a (finitary) $\LSigma[A]$-formula.
    \item 
    Using the translation of \Cref{def:deriv-inf-lang} into $\langl[A]$,
    we can use Gandy's theorem to find (finitary) $\LSigma[A]$ formulas describing the
    relation $\ttt\vdash^A$. Similarly, we can find $\Sigma[A]$ formulas defining
    the class $\LSigma[A]$ and the relation $\models^A_\Sigma$ within $\kp[A]$,
    and hence within $\ttt$, on the codes of formulas.
    \item An easy fact to notice, which we have implicitly already used,
    is that $\kp[A]$ can code itself in an $\LSigma[A]$ fashion. The same holds for $\ttt$,
    by the fact that we chose this extension to be given by a $\Sigma^A$-set of formulas.
\end{itemize}


Recall that in \citep[Section~I.5]{barwise} Barwise develops a theory showing that definitional extensions of $\kp$ by $\Sigma$-definable functions and $\Delta$-definable predicates are well-behaved.
An important point is that the addition of $\Sigma$-definable functions and $\Delta$-definable predicates to the signature of $\kp$ does not change the class $\Sigma$.

We can develop an analogous construction for the case of our infinitary theories.
Namely, for each $\Sigma[A]$-formula $\varphi(\vec{x},y)$ such that \begin{equation}\label{func_symb_def_eq}\ttt\vdash^A \forall \vec{x}\exists y \varphi(\vec{x},y)\land \forall \vec{x}\forall y_1,y_2(\varphi(\vec{x},y_1)\land \varphi(\vec{x},y_2)\to y_1=y_2),\end{equation}
we expand the language $\langl[A]$ by a function symbol $f$ defined by this formula.
Analogously, for any pair of $\Sigma[A]$-formulas $\varphi(\vec{x}),\psi(\vec{x})$ such that \begin{equation}\label{pred_symb_def_eq} \ttt\vdash^A \forall \vec{x}( \varphi(\vec{x})\mathrel{\leftrightarrow} \psi(\vec{x})),\end{equation}
we expand the language $\langl[A]$ by a predicate symbol $P(\vec{x})$.
Together with the addition of the function symbol we expand the provability notion to
$\ttt\vdash^A_+ $, which covers the extended language. The inductive definition of $\ttt\vdash^+ $ extends the definition of $\ttt\vdash$ by extending all the clauses to the extended language and adding extra axioms for the new function and predicate symbols. For each function symbol $f$ as above, the extra axiom is
\[ \ttt\vdash^A_+ f(\vec{x})=y \mathrel{\leftrightarrow} \varphi(\vec{x},y).\]
The extra axiom for a given predicate symbol $P$ is
\[\ttt\vdash^A_+ P(\vec{x}) \mathrel{\leftrightarrow} \varphi(\vec{x}).\]

Using standard arguments adapted to the infinitary case, we can naturally show that this extended provability notion satisfies two important properties. First, it does not yield any extra theorems in the original $\langl[A]$. Second, we can naturally define the extended class $\Sigma^+[A]$, which differs from $\Sigma[A]$ by allowing terms with the extra function symbols in both quantifier bounds and atomic formulas. And then for each $\Sigma^+[A]$ formula $\varphi$ there is a $\Sigma[A]$-formula $\varphi'$ such that $\ttt\vdash^A_+\varphi\mathrel{\leftrightarrow}\varphi'$.

In the above we have not provided any details on how the coding of the extra function symbols is done. In fact, to work with it in a smooth way it is desirable to ensure that the signature is $\Delta^A$, which means that the class of function symbols is $\Delta^A$, the class of predicate symbols is $\Delta^A$, and the function assigning the arity to a symbol is $\Sigma^A$. If the signature is not $\Delta^A$ in this sense, then the class of formulas will not be $\Delta^A$, causing extra issues. Thus, we clarify the definition above as follows. When we add a function symbol $f$ it should in fact be treated as a member of a family of function symbols $f_{\varphi,p}$, where $p$ is a proof of (\ref{func_symb_def_eq}) (in particular, this means that for each $\varphi$ that indeed defines a function we add a proper class of extra function symbols). Similarly, $P$ should in fact be coded as $P_{\varphi,\psi,p}$, where $p$ is a proof of (\ref{pred_symb_def_eq}).  With this clarification, we can naturally define a $\Sigma^A$ function transforming a given formula $\varphi\in \Sigma^+[A]$ to an equivalent $\Sigma[A]$ formula $\varphi'$, and furthermore a $\Sigma^A$ function providing a corresponding $\ttt$-proof of the equivalence between $\varphi$ and $\varphi'$.

Thus, we can in fact freely use $\Sigma[A]$-definable functions and $\Delta[A]$-definable predicates within $\Sigma[A]$-formulas. Formally, this will of course entail the elimination of definitions when necessary. But fundamentally this matches the usual conventions used in the study of provability in arithmetic.

\begin{lemma}
  For every variable $v_0$ of $\langl[A]$, there is a $\kp[A]$-provable $\Sigma[A]$-definable binary ``substitution'' function $s_{v_0}$ with the following properties:
  For every formula $\varphi$ of $\langl[A]$ and constant $\Godelnum{a}$, we have that $\ttt\vdash^A s_{v_0}(\Godelnum{\varphi(v_0)},  a )= \Godelnum{\varphi(\Godelnum{a})}$.
\end{lemma}
\begin{proof}
    It is easy to define by recursion a function $s_{v_0} (x,y)$ such that $s(x,y)=\emptyset$ if $x$ is not the code of any formula of $\langl[A]$, and $s(x,y)=x[v_0/\Godelnum{y}]$ otherwise.
\end{proof}

Naturally we then introduce for every finite list of variables $\vec{v}=\langle v_0,\ldots,v_{n-1}\rangle $ a ``multi-substitution'' function \[s_{\vec{v}}(x,y_0,\ldots,y_{n-1})=s_{v_0}(\dots s_{v_{n-1}}(x, y_{n-1}),\dots y_0).\]

\begin{definition}
    For every formula $\varphi$ of $\langl[A]$ and a list of free-variables $\vec{v}=\langle v_0,\ldots,v_{n-1}\rangle$  we denote by $\Godelnum{\varphi(\dot v_0,\dots, \dot v_{n-1})}$ the term \[s_{\vec{v}}(\Godelnum{\varphi(v_0,\dots,v_{n-1})},v_0,\ldots,v_{n-1}).\]
\end{definition}


Our next goal is to show that the satisfaction relation $\models^A_\Sigma$ for $\LSigma[A]$ formulas is well-behaved.

Let $\models_{\Sigma}$ be the internal version of $\models_\Sigma^A$ that we can use in extensions of $\kp[A]$. That is it is a $\Sigma$-formula (and hence $\Sigma[A]$-formula) $\models_\Sigma^V$. Namely, $\models_{\Sigma}$ is the result of the construction of $\models^A_\Sigma$ in $\kp$ with $A$ interpreted as $V$. Naturally, $\kp+\mathsf{Tr}(A)+\kp^A$ can prove that the relativization of $\models_\Sigma$ to $A$ is equivalent to $\models_\Sigma^A$. 

Analogously we introduce the relation $T\vdash$ as the internalized version of $T\vdash^A$. Note that for our $A$-theories we allow parameters from $A$ in the $\Sigma^A$ formulas defining their sets of axioms. For the internalized version of the relation we substitute this parameters $p_0,\ldots,p_{n-1}$ as the corresponding constants $\Godelnum{p_0},\ldots,\Godelnum{p_{n-1}}$. Finally, to better match with the standard notations we also denote the $\Sigma[A]$ formula $T\vdash x$ as $\prov(x)$.

\begin{lemma}\label{lem:coded-truth}
    For every $\LSigma[A]$ formula $\varphi$ whose set of free variables is finite and exhausted by a finite list $\vec{x}$,
    $\ttt\vdash^A \varphi(\vec{x}) \biimp (\models_\Sigma\Godelnum{\varphi(\dot{\vec{x}})})$.
\end{lemma}
\begin{proof}
    The proof is by induction on the construction of the formula $\varphi$.

    Suppose that $\varphi$ is atomic. The crucial observation for this case is that,
    in the inductive definition of $\models_\Sigma \varphi$, all the literals are decided in the first step.
    Note that each literal $\varphi(\vec{x})$  is obtained from either $x_1\in x_2$,
    $\lnot x_1\in x_2$, $x_1=x_2$, or $\lnot x_1=x_2$ by perhaps renaming variables and
    substituting constants for them. And for each of these it is easy to see that,
    regardless of the choice of a finitary definable (perhaps with parameters) class $X$,
    the theory $\ttt$ can prove that for any $\vec{x}$, the element
    $\Godelnum{\varphi(\dot {\vec{x}})}$ lies in $\Gamma_{\models_\Sigma}(X)$
    iff $\varphi(\vec{x})$. Applying this to the case
    $X=\{\Godelnum{\varphi}\colon \models_\Sigma \Godelnum{\varphi}\}$,
    we conclude that $\ttt\vdash^A\forall \vec x(\varphi(\vec x)
    \biimp (\models_\Sigma\Godelnum{\varphi(\dot{\vec x})}))$ holds in this case.

    We prove the claim in the case that $\varphi(\vec{x})$ is of the form $\bigwedge \Psi(\vec{x})$ for a set
    of $\LSigma[A]$-formulas $\Psi(\vec{x})$. By fixed point condition for $\models_{\Sigma}$,
    \[\ttt\vdash^A(\models_\Sigma\Godelnum{\bigwedge\Psi(\dot{\vec{x}})} )
    \biimp \forall y\in \Godelnum{\Psi(\vec{x})} (\models_\Sigma s_{\vec{x}}(y,\vec{x}))).\]
    Thus, in this case, it suffices to show that
    \begin{equation}
    \label{coded-truth-conj-equiv}
    \ttt\vdash^A\bigwedge\Psi(\dot{\vec{x}})
    \biimp \forall y\in \Godelnum{\Psi(\vec{x})} (\models_\Sigma s_{\vec{x}}(y,\vec{x}))).
    \end{equation}

    By the induction hypothesis, $\ttt\vdash^A\psi(\vec{x})\mathrel{\leftrightarrow} (\models_\Sigma \Godelnum{\psi(\dot{\vec{x}})})$ for all $\psi(\vec{x})\in \Psi(\vec{x})$. Hence, for all $\psi(\vec{x})\in \Psi(\vec{x})$, via a finitary first-order deduction in $\ttt$, we conclude that \[\ttt\vdash^A\psi(\vec{x}) \to (y=\Godelnum{\psi(\vec{x})}\to (\models_\Sigma s_{\vec{x}}(y,\vec{x}))).\]
    Next, via a straightforward infinitary propositional derivation we get \[\ttt\vdash^A \bigwedge\Psi(\vec{x}) \to \Big(\big(\bigvee_{\psi(\vec{x})\in \Psi(\vec{x})} y=\Godelnum{\psi(\vec{x})}\big)\to (\models_\Sigma s_{\vec{x}}(y,\vec{x}))\Big ).\]
    Then via the Bernays rule for $\forall$ we get
    \[\ttt\vdash^A \bigwedge\Psi(\vec{x}) \to \forall y \Big(\big(\bigvee_{\psi(\vec{x})\in \Psi(\vec{x})} y=\Godelnum{\psi(\vec{x})}\big)\to (\models_\Sigma s_{\vec{x}}(y,\vec{x}))\Big ).\]
    Now note that axioms \ref{R_A_ax3}.\ and \ref{R_A_ax1}.\ of $\mathsf{R}[A]$ together tell us that \[\mathsf{R}[A]\vdash^A (\bigvee_{\Godelnum{b}\in \Godelnum{a}} y=\Godelnum{b}) \mathrel{\leftrightarrow} y\in \Godelnum{a}.\] Thus we can equivalently transform $\bigvee_{\psi(\vec{x})\in \Psi(\vec{x})} y=\Godelnum{\psi(\vec{x})}$ into $y\in \Godelnum{\Psi(\vec{x})}$, and thus we conclude
    \[\ttt\vdash^A \bigwedge\Psi(\vec{x}) \to \forall y\in \Godelnum{\Psi(\vec{x})} \,(\models_\Sigma s_{\vec{x}}(y,\vec{x})),\]
    which gives the left-to-right implication in the desired equivalence (\ref{coded-truth-conj-equiv}).

    To prove the right-to-left part of (\ref{coded-truth-conj-equiv}), we reason as follows. For each individual $\psi(\vec{x})\in \Psi(\vec{x})$, using axiom \ref{R_A_ax1}.\ of $\mathsf{R}[A]$ we conclude that
    \[\ttt\vdash^A \forall y\in \Godelnum{\Psi(\vec{x})} (\models_\Sigma s_{\vec{x}}(y,\vec{x}))\to \models_{\Sigma} \Godelnum{\psi(\dot \vec{x})},\]
    and next, using the induction hypothesis, we get
\[\ttt\vdash^A \forall y\in \Godelnum{\Psi(\vec{x})} (\models_\Sigma s_{\vec{x}}(y,\vec{x}))\to \psi(\vec{x}).\]
    Thus, via the rule for conjunction, we get the desired implication
    \[\ttt\vdash^A \forall y\in \Godelnum{\Psi(\vec{x})} (\models_\Sigma s_{\vec{x}}(y,\vec{x}))\to \bigwedge \Psi(\vec{x}).\]

    Consider the case where $\varphi$ has the form $\exists y\psi(y,\vec x)$.
    Showing that \[\ttt\vdash^A \forall \vec x
    ((\models_\Sigma \Godelnum{\exists y\psi(y,\dot{\vec x})})\imp \exists y\psi(y,\vec x))\] is easy,
    since $\ttt$ knows that for $\models_\Sigma \Godelnum{\exists y\psi(y,\dot{\vec x})}$ to
    hold there must be a constant $\Godelnum{a}$ such that $\models_\Sigma \Godelnum{\psi(\Godelnum{a},\dot{\vec x})}$,
    and we can conclude by the inductive hypothesis. We are thus left to show that
    \[\ttt\vdash^A \forall \vec x (\exists y\psi(y,\vec x)\imp
    (\models_\Sigma \Godelnum{\exists y\psi(y,\dot{\vec x})})).\]
    By the inductive hypothesis, we have that \[\ttt\vdash^A \forall \vec x,y (\psi(y,\vec x)\imp
    (\models_\Sigma \Godelnum{\psi(\dot y,\dot{\vec x})})).\]
    It is then sufficient to notice that
    \[\ttt\vdash^A \forall \vec x,y ((\models_\Sigma \Godelnum{\psi(\dot y,\dot{\vec x})}) \imp
    (\models_\Sigma \Godelnum{\exists y\psi(y,\dot{\vec x})})),\]
    by the way the existential is introduced within the relation $\models_\Sigma$.
    We can then conclude by logic.

     The cases of disjunction and bounded quantification can be handled by similar reasoning to the above two cases.
\end{proof}

\begin{remark}\label{rem:eq-sig-sig}
    Notice that \Cref{lem:coded-truth} has the following important consequence: over $\kp[A]$ every $\LSigma[A]$ formula is equivalent to a $\Sigma^A$ (with a parameter from $A$, namely the targeted $\LSigma[A]$-formula itself). This follows from the fact that $\models_\Sigma$ is actually a \emph{finitary} formula. This in particular entails that $\kp[A]$ proves the schemes of $\Sigma[A]$-Collection and $\Delta[A]$-Separation.
\end{remark}

An analogous argument allows us to recover the first derivability condition.
\begin{lemma}\label{lem:nec-rule}
    For any formula $\varphi$ of $\langl[A]$,
    if $\ttt\vdash^A \varphi$, then $\ttt\vdash^A \prov(\Godelnum{\varphi})$.
\end{lemma}
\begin{proof}
    Since $\mathsf{Pr}$ is a $\Sigma[A]$-formula, by \Cref{lem:truth-imp-prov}, the only thing we have to prove is that
    $\ttt\vdash\varphi$ implies $\models^A_\Sigma \prov(\Godelnum{\varphi})$.

    This implication is a particular case of the following fact about our meta-theory, whose validity we will explain next.
    For every fixed $\Sigma$-sentence $\psi(\vec{x})$, within $\kp+\mathsf{Tr}(A)+\kp^A$ we can prove the equivalence between $\psi^A(\vec{x})$ and $\models_\Sigma^A \psi(\dot{\vec{x}})$. Note that here, when we write $\models_\Sigma^A \psi(\dot{\vec{x}})$, the role of $\psi$ is different from when we just write $\psi$. Namely, $\psi(\dot{\vec{x}})$ in  $\models_\Sigma^A \psi$ stands for a natural $\Sigma$-term that maps values of $\vec{x}$ to the $\langl[A]$-formula $\psi(\Godelnum{x_0},\ldots,\Godelnum{x_{n-1}})$.

    The fact above can be proved via an external (i.e.~outside of $\kp+\mathsf{Tr}(A)+\kp^A$) induction on the construction of $\Sigma^A$ formulas. Each step of the induction simply proceeds by unpacking the fixed-point property of $\models_\Sigma^A$ (similarly to the proof of \Cref{lem:coded-truth}).
\end{proof}

This was the main step towards proving that the standard Hilbert-Bernays derivability
conditions hold for our provability predicate; we complete the proof of this fact in the
following Theorem.

\begin{theorem}\label{theo:deriv-ax}
    $\prov$ satisfies the Hilbert-Bernays-L\"ob derivability conditions, namely:
    \begin{enumerate}
        \item[HBL1] For any formula $\varphi$ of $\langl[A]$,
        if $\ttt\vdash\varphi$, then $\ttt\vdash^A \prov(\Godelnum{\varphi})$.
        \item[HBL2] For any sentences $\varphi, \psi\in\langl[A]$, we have that
        $\ttt\vdash^A \prov(\Godelnum{\varphi\imp\psi}) \imp
        (\prov(\Godelnum{\varphi}) \imp \prov(\Godelnum{\psi}))$.
        \item[HBL3] For every $\LSigma[A]$ sentence $\varphi$,
        $\ttt\vdash\varphi\imp \prov(\Godelnum{\varphi})$. In particular,
        $\ttt\vdash^A (\prov(\Godelnum{\psi})\imp
        \prov(\Godelnum{\prov(\Godelnum{\psi})}))$
        for any $\langl[A]$ sentence $\psi$.
    \end{enumerate}
\end{theorem}
\begin{proof}
    HBL1 is proved in \Cref{lem:nec-rule}.

    HBL2 simply asserts that, provably in $\ttt$, the notion of provability is closed under modus ponens. Note that $\ttt$ internally knows that $\mathsf{Pr}$ is a fixed point of the inductive definition of provability. Thus, in particular $\ttt$ proves the closure under the rule of modus ponens.

    For HBL3: one easily sees that the proof of \Cref{lem:truth-imp-prov} can be formalized in $\ttt$ (as discussed earlier, this is achieved by interpreting our metatheory in $\ttt$, where we keep $\in$ without change and put $A=V$). Thus, we get that $\ttt\vdash^A (\models_\Sigma \Godelnum{\varphi}) \imp \prov(\Godelnum{\varphi})$. The conclusion then follows from \Cref{lem:coded-truth}.
\end{proof}

Intuitively speaking, the results so far would allow us to conclude that
the provability logic associated with $\ttt$ is at least $\mathsf{K4}$. But we can do better.




\begin{lemma}[Gödel's Diagonal Lemma]
    For every formula $\varphi(v)\in \langl[A]$, perhaps with extra free variables, there exists a formula $\psi$ of $\langl[A]$ such that $\ttt\vdash^A \varphi(\Godelnum{\psi}) \biimp \psi$.

    Moreover, if $\varphi$ is $\LSigma[A]$, then so is $\psi$.
\end{lemma}
\begin{proof}


    The proof goes as in the classical case. First we define the self-substitution function $f(x)=s_{v}(x,x)$. Clearly, for every formula $\theta(v)$, we have that
    \[
    \ttt\vdash^A s(\Godelnum{\theta(v)})=
    \Godelnum{\theta(\Godelnum{\theta(v)})}.
    \]

    Then, given our formula $\varphi$, we define the auxiliary formula $\chi(v)$ as the formula $\varphi(f(v))$. Finally, we let our required $\psi$ to be  $\chi(\Godelnum{\chi(v)})$.
    Then, one has
    $$
    \begin{aligned}
        T\vdash\psi & \biimp & \chi(\Godelnum{\chi(v)})  \\
        & \biimp &\varphi(s(\Godelnum{\chi(v)}))  \\
        & \biimp & \varphi(\Godelnum{\chi(\Godelnum{\chi(v)})}) \\
        & \biimp& \varphi(\Godelnum{\psi}).
    \end{aligned}
    $$

    The fact that $\psi$ is $\LSigma[A]$ follows immediately from the assumption that $\varphi$ is.
\end{proof}

Trivially, we can derive the following corollary of the Diagonal Lemma that will be convenient to be used later in Section \ref{sec:solovay}:
\begin{corollary}
    For every number $n$ and formula $\varphi(v,x_1,\dots, x_n)$, with at most the highlighted variables as free variables, there exists a formula $\psi(x_1,\dots,x_n)$ of $\langl[A]$ such that \[\ttt\vdash^A \varphi(\Godelnum{\psi( \dot x_1,\dots,\dot x_n)},x_1,\dots,x_n) \biimp \psi(x_1,\dots,x_n).\]

    Moreover, if $\varphi$ is $\LSigma[A]$, then so is $\psi$.
\end{corollary}

\begin{remark}
    For subtheories of arithmetic, it is well known that a version of Gödel's Diagonal
    Lemma can be proved already in $\mathsf{Q}$ (and in fact in $\mathsf{R}$); it is therefore natural to expect that something
    similar should happen here. In the present paper we do not try to optimize the assumptions on the theory needed to prove
    the Lemma. We point out that we do not see how to immediately adapt the proof that we give for $\kp[A]$ to the case of, for instance, the theory $\mathsf{R}[A]$. The main issue is the representability of the
    self-substitution function in weaker theories. In the arithmetical case it is a standard fact that every
    recursive function $f\colon \mathbb{N}\to \mathbb{N}$ is representable in $\mathsf{Q}$, and this general fact is then applied to the self-substitution function in particular. But in our case, it is unclear whether the analogue of this can be achieved for arbitrary admissible sets $A$. And in the proof for the case of $\kp[A]$ we use the fact that the self-substitution function is provably total in $\kp[A]$. Thus, it seems implausible that a generalization of the classical technique from the arithmetical case would work for weak theories (analogous to $\mathsf{Q}$ in our setting).

    We add, however, that, based on unpublished work by the second author about a different approach to proving the Diagonal Lemma in the case of finite languages, it seems likely that
    $\mathsf{R}[A]$ can indeed prove a version of the Diagonal Lemma (with a rather different argument).
\end{remark}

The question about theories representing all $\Sigma^A$ functions mentioned in the remark above can be made precise as follows.
\begin{question}
Is it true that for every admissible set $A$ there is a $\Sigma^A$-theory $T_0$ such that for every total $\Sigma^A$-function $f(x)$ on $A$ there is $G_f(x,y)\in\langl[A]$ for which \[T_0\vdash^A G_f(\Godelnum{a},\Godelnum{f(a)})\;\;\;\text{and}\;\;\; T_0\vdash^A  G_f(\Godelnum{a},x)\to x=\Godelnum{f(a)}\text{, for every }a\in A?\]
\end{question}

Gödel's Diagonal Lemma allows us to prove the following in a manner identical to the arithmetical case:

\begin{corollary}\label{cor:lobs-ax}
The theory $\ttt$ satisfies Löb's Theorem:
\[\ttt\vdash^A \mathsf{Pr}(\Godelnum{\varphi})\to\varphi \mathrel{\Rightarrow} \ttt\vdash^A \varphi,\text{ for each sentence }\varphi\in\langl[A].\]
Furthermore  $\ttt$ satisfies the formalized Löb's Theorem:
\[\ttt\vdash^A \mathsf{Pr}(\Godelnum{\mathsf{Pr}(\Godelnum{\varphi})\to\varphi})\to \mathsf{Pr}(\Godelnum{\varphi}),\text{ for each sentence }\varphi\in\langl[A].\]
\end{corollary}

The results we have obtained so far are enough to show that the provability logic associated
with $\ttt$ is at least infinitary $\mathsf{GL}$ without Barcan's formula. Next, we show that $\ttt$ is also strong enough to
prove Barcan's formula.

\begin{lemma}\label{lem:barcan}
    For every set $\Phi\in A$ of $\langl[A]$ formulas, we have that
    $\ttt\vdash^A \bigwedge_{\varphi\in\Phi} \prov(\Godelnum{\varphi})
    \imp \prov(\Godelnum{\bigwedge_{\varphi\in\Phi} \varphi})$.
\end{lemma}
\begin{proof} Since, $\ttt$ knows that $\mathsf{Pr}$ is a fixed point of the inductive definiton of provability,
    $\ttt$ proves that $\prov(\Godelnum{\bigwedge_{\varphi\in\Phi} \varphi})$ is equivalent to $\forall x\in \Godelnum{\Phi}(\prov(x))$. To finish it thus suffices to show that $\ttt\vdash^A \bigwedge_{\varphi\in\Phi}
    \prov(\Godelnum{\varphi}) \imp \forall x\in \Godelnum{\Phi}(\prov(x))$, which we clearly can do using axioms of $\mathsf{R}[A]$.
\end{proof}

We are finally ready to present the desired lower bound for
the provability logic of theories $\ttt$ as in \Cref{def:kpa-ta}.

\begin{definition}
    Let us fix an element $\Box$ (say, in $\mathbb{HF}$) different from $\bigwedge, \bigvee, \imp$ and
    $\neg$, and let $P$ be a $\Delta^A$ class disjoint from $\{\Box,\bigwedge,\bigvee,\imp,\neg\}$
    (the elements of $P$ should be thought of as the propositional variables). We define the
    \emph{modal language} $\langl[A]^{\Box}$ as the smallest $\Sigma^A$ class containing every
    element of $P$ and closed under applications of $\Box$, $\to$, $\neg$ to elements of
    $\langl[A]^{\Box}$ and under application of $\bigvee$ and $\bigwedge$ to sets of
    elements of $\langl[A]^{\Box}$.
\end{definition}

Next we introduce the logic $\hgl[A]$ that is the natural counterpart of $\HGL{\omega_1}$ for the case of the modal language $\langl[A]^\Box$.
\begin{definition}
    We define the relation $\hgl[A]\vdash$ over the class $\langl[A]^{\Box}$
    as the smallest class such that:
    \begin{enumerate}
        \item $\hgl[A]\vdash \textit{S}$ holds for each instance $\textit{S}$ of one of the axioms  $1$ to $6$ of $\HCLA$, where formulas range over the language $\langl[A]^{\Box}$;
        \item $\hgl[A]\vdash \textit{S}$ holds for each instance $\textit{S}$ of one of the axioms \ref{GL_ax_first}.--\ref{GL_ax_last}. of $\HGL{\omega_1}$, where formulas range over the language $\langl[A]^{\Box}$;
  \item if $\hgl[A]\vdash B$ and $\hgl[A]\vdash B \imp C$, then $\hgl[A]\vdash C$
  (modus ponens);
  \item for every set of $\langl[A]^\Box$-formulas $X\in A$ and $B\in X$, we have the axiom $\bigwedge X\to B$ (conjunction axioms);
  \item if $\hgl[A]\vdash B\imp C$ for every $C$ in a set of $\langl[A]^{\Box}$ formulas
  $X\in A$, then $\hgl[A]\vdash B\imp\bigwedge X$ (conjunction rule);
  \item for every set of $\langl[A]^{\Box}$-formulas $X\in A$ and $B\in X$, we have the axiom $B\to\bigvee X$ (disjunction axioms);
  \item if $\hgl[A]\vdash B\imp C$ for every $B$ in a set of $\langl[A]^{\Box}$ formulas
  $X\in A$, then $\hgl[A]\vdash \bigvee X \imp C$ (disjunction rule);
  \item if $\hgl[A]\vdash B$, then  $\hgl[A]\vdash \Box B$ (necessitation)
    \end{enumerate}
\end{definition}

\begin{remark}
    The relation $\hgl[A]\vdash$ on modal formulas is obviously very similar to the system
    $\HGL {\omega_1}$ used in the previous sections. We cannot,
    however, simply copy-paste its definition in the present setting:
    the same considerations expressed in \Cref{rem:proof-system} apply here as well.
    Namely, the $\Sigma^A$ relation defined above does not, in general, correspond neatly
    to the standard notion of provability given by proof trees. We need a more relaxed
    notion of proof in this case as well.
\end{remark}

\begin{definition}
    An \emph{interpretation} $(\cdot)^*$ of $\langl[A]^{\Box}$ into $\langl[A]$
    is a class function from $\langl[A]^{\Box}$ to the sentences of $\langl[A]$ such that
    \begin{itemize}
        \item $(\cdot)^*$ commutes with the finitary and infinitary Boolean connectives (e.g., if $B$ is of the form
        $\neg C$, then $B^*$ is $\neg (C^*)$);
        \item if $B$ is of the form $\Box C$, then $B^*$ is $\prov ( \Godelnum{C^*})$.
    \end{itemize}
\end{definition}

\begin{theorem}\label{theo:low-bound-kp}
    Let $\ttt$ be a $\Sigma^A$ theory of $\langl[A]$ extending $\kp[A]$.
    Then, for every interpretation $(\cdot)^*$ of $\langl[A]^{\Box}$ into $\langl[A]$
    and every formula $B$ in $\langl[A]^{\Box}$, we have that if $\hgl[A]\vdash B$,
    then $\ttt\vdash^A  B^*$.
\end{theorem}
\begin{proof}[Sketch of proof]
    The proof is in essence the same as the one for the case of arithmetic,
    as given, for instance, in Theorem 2 of Chapter 3 of \cite{logic-prov-boolos}.

    More explicitly, we fix an interpretation $(\cdot)^*$ and then proceed by induction on the inductive definition of $\hgl[A]\vdash$ to show that $\hgl[A]\vdash B$ entails $\ttt\vdash^A  B^*$. Namely, by the minimality of $\hgl[A]\vdash$, it is sufficient to prove that the class $O_*=\{ B\mid \ttt\vdash^A  B^*\}$ contains all the axioms of $\hgl[A]$ and is closed under all the inference rules of the logic.

    Non-modal axioms and rules are shared between the two calculi, and hence all $\langl[A]^\Box$-instances of those axioms will be in $O_*$ and $O_*$ will be closed under those rules. The fact that $\ttt$ satisfies HBL1 corresponds to the closure of $O_*$ under the necessitation rule. HBL2 and HBL3 entail that $O_*$ contains all the $\langl[A]^\Box$-instances of axioms $K$ and $4$, respectively. As we already mentioned in \Cref{cor:lobs-ax}, the formalized L\"ob's Theorem holds for $\ttt$, which means that all $\langl[A]^\Box$-instances of L\"ob's axiom will be in $O_*$. And finally, as we proved in \Cref{lem:barcan}, all $\langl[A]^\Box$-instances of Barcan's formula are in $O_*$.
\end{proof}



\section{An analogue of Solovay's completeness theorem}  \label{sec:solovay}



The goal of this section is to show that, under some special circumstances, the system
$\dgla$ is complete with respect to the semantics provided by the interpretations $(\cdot)^*$
defined above (though we do not know whether $\dgla$ is sound for this interpretation). Namely, we will show that for a sentence $B$ in $\langl[A]^{\Box}$,
if $\dgla\not\vdash B$, then we can find an interpretation $(\cdot)^*$ of
$\langl[A]^{\Box}$ into $\langl[A]$ such that $\ttt\not\vdash^A  B^*$.

We will follow a two-step argument to achieve this.
\begin{enumerate}
    \item We first try to adapt to our setting the arithmetical completeness theorem for $\mathsf{GL}$,
    due to Solovay. This gives us that, under some assumptions on $A$ and $\ttt$,
    if we can find a Kripke model $\mathcal{K}\in A$ such that $\mathcal{K}\not\models B$,
    then we can also find an interpretation $(\cdot)^*$ such that $\ttt\not\vdash^A  B^*$.
    This is done in \Cref{theo:km-compl}. As we will see, what we need for the proof to
    go through is that, on the one side, $A$ is a set, and on the other, that
    $\ttt$ proves the existence of an $\LSigma[A]$ surjection from
    the ordinals onto the universe. Intuitively, this is due to the fact that, in our provability
    predicate $\exists x\proov(x,\Godelnum{\varphi})$, the witness $x$ can be any set,
    which means that in our setting the class of proofs is not automatically indexed by ordinals,
    unlike the case of $\mathsf{PA}$, where everything is coded by natural numbers. The latter is important
    since an examination of Solovay's proof shows that it relies on the fact that the theory under
    consideration is aware of the presence of a well-ordering of proofs that is furthermore suitably
    aligned with the formalized provability predicate.
    \item The previous result allows us, in principle, to use \Cref{dgla-completeness}
    to recover the desired completeness result. We do this in \Cref{theo:compl-di-syst},
    under some further conditions on $A$ and $\ttt$ (and on the metatheory in which we are working):
    we require that $A=L_\alpha$ for
    a countable stable ordinal $\alpha$, and that $\ttt\vdash^A  V=L$. Intuitively,
    the first condition allows us to ``import'' the countermodel $\mathcal{K}$ built in
    \Cref{dgla-completeness} into $A$, and the second condition guarantees that $\ttt$
    proves the existence of a surjection as needed for \Cref{theo:km-compl}.
\end{enumerate}

We can then view the combination of \Cref{theo:low-bound-kp} and \Cref{theo:compl-di-syst} as
giving lower and upper bounds for the provability logic of some special $A$ and $\ttt$.
We remark, however, that these bounds need not coincide, as far as we know.

\bigbreak




Before moving to the proof of the main results of this section, notice that we still
have some setup to do; indeed, technically speaking, we still do not know what a Kripke model is in our setting, and similarly, what it means for a Kripke
model to satisfy a formula in the modal language $\langl[A]^{\Box}$.

In what follows, it will be important that Kripke models be elements of our admissible $A$. Thus, we put some restriction
on the relation $\EV$ and require that it come equipped with a set $S$ specifying its support, i.e., the set of
propositional variables on which it is defined.

\begin{definition}
    For every formula $B$ in the modal language $\langl[A]^{\Box}$, we define the \emph{support} of $B$ to be the set $S\subseteq P$ of propositional variables that appear in $B$ (the fact that the support is indeed a set follows immediately from the fact that $B$ is a set).

    We call a \emph{Kripke model} a $4$-tuple $\kmodel$ with the following properties:
    \begin{itemize}
    \item $\mathcal{K}\in A$.
    \item $\sqsubset$ is a binary transitive relation on $W$.
    \item $S$ is a subset of the class $P$ of propositional variables.
    \item $\EV$ is a subset of $W\times S$.
    \end{itemize}

    Given a Kripke model $\kmodel$, a node $w\in W$, and a propositional variable $p\in S$,
    we say that $\mathcal{K}, w\models p$ if ${\langle w,p\rangle}\in {\EV}$, and $\mathcal{K}, w\not\models p$ otherwise. Given a node $w$, we can extend the notions of satisfaction and non-satisfaction
    $\mathcal{K},w\models B$ and $\mathcal{K},w\not\models B$ to all formulas $B$ in $\langl[A]^{\Box}$ with support contained in $S$,
    following the standard definition of satisfaction given in \cref{sec-gla-kripke}.

    Finally, for $\mathcal{K}$ and $B$ as above, we say that $\mathcal{K}\EV B$ if
    $\mathcal{K}, w\EV B$ holds for all $w\in W$.
\end{definition}
Notice that, in the definition above, given a Kripke model $\kmodel$, the relations $\mathcal{K}\EV B$ and $\mathcal{K}\not\EV B$ are restricted to formulas $B$ whose support is contained in $S$.

As usual, we will say that $\kmodel$ is conversely well-founded if $(W,\sqsubset)$ is.

Notice that, thanks to the assumption that $\mathcal{K}\in A$, the relation $\mathcal{K}\models B$ is $\Delta_1$ with parameters from $A$.


\begin{definition}
Let $(W,\sqsubset)$ be a relational structure. A \emph{co-ranking function} for the relation $\sqsubset$ on $W$ is a function $f$ from $W$ to the ordinals such that for all $u \in W$:
$$f(u) = \sup_{v\sqsupset u}(f(v)+1)$$
\end{definition}

The next lemma will allow us to import an ``external'' co-rank on $W'$ into $A$. But in order to do this, we will have to put the first limitation on the admissible classes we consider: namely, from here on, $A$ is assumed to be an admissible \emph{set}.

\begin{lemma}\label{lem:global-coranking}
Suppose $A$ is a set. If $(W,\sqsubset)\in A$ is a conversely well-founded binary relation, then there exists a global co-ranking function for $W$ within $A$.
\end{lemma}

\begin{proof}
We say that a function $f$ from a subset of $W$ to the ordinals is a \emph{partial co-rank} if its domain is $\sqsubset$-upward-closed (i.e., if $u,v\in W$, $u\in \dom(f)$ and $u\sqsubset v$, then $v\in \dom(f)$), and $\forall u\in \dom(f) (f(u) = \sup_{v\sqsupset u}(f(v)+1))$.

First, observe that any two partial co-ranks have the same values on the intersection of their domains: if not, then we use the converse well-foundedness of $\sqsubset$ to find a maximal point of divergence, and trivially obtain a contradiction.

Next, we prove by contradiction that each $u\in W$ is in the domain of some partial co-rank $f\in A$.

Suppose for a contradiction that this does not hold. We want to consider the set of all $u\in W$ that are not in the domain of any partial co-rank, and consider a minimal element of this set, as given by the well-foundedness of $W$. In any case, the definition of this set would require $\Sigma^A$-separation in general, which can be performed since $A$ is a set, which makes $\Sigma^A$-formulas $\Delta_0$-formulas with $A$ as a parameter.

Since all successors of $u$ are in the domain of some partial co-rank, using $\Sigma^A$-collection in $A$ we find a set $F\in A$ of partial co-ranks such that each successor of $u$ is in the domain of some $f\in F$. Note that, since all partial co-ranks agree with each other,  $f_u=\bigcup F$ is a partial co-rank. We extend $f_u$ to $f_u^+$ by setting $f_u^+(u)=\sup_{v\sqsupset u}f_u(v)+1$. Since $f_u^+$ is clearly a partial co-rank having $u$ in its domain, we obtain a contradiction.

Finally, we use $\Sigma^A$-collection on $A$ to find a set $F\in A$ of partial co-ranks such that for each $u\in W$ there is $f\in F$ with $u\in \dom(f)$. Clearly, $\bigcup F\in A$ is a partial co-rank having all of $W$ as its domain, i.e., a global co-ranking function for $W$.
\end{proof}

In what follows, we will need the generalization of $\Sigma$-soundness to theories in the language $\langl[A]$.
\begin{definition}
    Let $\ttt$ be an $A$-theory. We say that $\ttt$ is \emph{$\LSigma[A]$-sound} if, for every $\LSigma[A]$ sentence $\varphi$, $\ttt\vdash^A  \varphi$ implies $\models_\Sigma^A \varphi$.
\end{definition}

\begin{theorem}\label{theo:km-compl}
    Suppose $A$ is a set. Let $\ttt$ be a $\Sigma[A]$-sound $A$-theory extending $\kp[A]$ such that there is a $\ttt$-provable $\Sigma[A]$-definable function $e(x)$ for which $\ttt$ proves that its restriction to the ordinals is a surjection onto $V$. Let $B$ be a formula in the
    modal language $\langl[A]^{\Box}$ such that there is a conversely well-founded Kripke model
    $\mathcal{K}\in A$ with $\mathcal{K}\not\models B$. Then there is also an interpretation
    $(\cdot)^*$ such that $\ttt\not\vdash^A  B^*$.
\end{theorem}

\begin{proof}
We start the proof with two technical observations.

The first is that, since the existence of the co-ranking function $\mathsf{rk}$ for $(W,\sqsubset)$ is a $\Sigma$-sentence that is true in $A$, it is also provable in $\ttt$. The existence of $\mathsf{rk}$ has important corollaries. We say that $f\colon \alpha\to W$ is a $\sqsubseteq$-\emph{weakly ascending} function if $\alpha\in\mathsf{On}$ and $\forall \beta,\gamma\in \alpha(\beta<\gamma\to f(\beta)\sqsubseteq f(\gamma))$. Since $\mathsf{rk}\circ f$ is an anti-monotone function from an ordinal to the class of ordinals, it must attain its least value at some $\beta\in \alpha$. Notice that $f$ must stabilize after this $\beta$, with $f(\gamma)\sqsubseteq f(\beta)$ for any $\gamma\in \alpha$. The same holds for weakly ascending class functions $f\colon \mathsf{Ord}\to W$. This reasoning about weakly ascending functions can also be formalized inside $\ttt$. In particular, we will be interested in the version of the theorem for $\ttt$-provable $\Sigma[A]$-definable functions in the case of class functions.

The second observation is that, as one can easily show,
if there exists a surjection $e$ as in the hypothesis, then, provably in $\ttt$,
there is also one that enumerates every element
of $A$ cofinally often; we will assume to be working with such an $e$.

The rest of the proof follows the classical proof of Solovay's theorem.
We are going to define an interpretation $(\cdot)^*$ such that $\ttt\not\vdash^A  B^*$.


Without loss of generality we may assume that the model $\mathcal{K}$ has a root $t$ and $\mathcal{K},t\not\models  B$: if not, then pick $t\in W$ such that $\mathcal{K},t\not\models  B$ and replace $\mathcal{K}$ with its restriction to the upper cone of $t$. We let $\mathcal{K}'=(W',\sqsubset,\EV,S)$ be $\mathcal{K}$ extended with a new root $r$.
Let $t$ be the root of $\mathcal{K}$; we extend the relation $\EV$ to $r$ by saying that,
for every $p\in S$,
$r\EV p$ holds if and only if $t \EV p$ does.
Notice that we are technically abusing notation by using the same symbols $\sqsubset$ and
$\models$ for the models $\mathcal{K}$ and $\mathcal{K}'$, but this should not cause any confusion,
since the relations coincide on common points of the domain $W$. We also naturally extend the co-ranking function $\mathsf{rk}$ to the root $r$.

The idea is, similarly to the arithmetical case, to define a function $h:\mathsf{Ord} \to W'$
that moves along $W'$ to a new point $w$ every time it finds a proof that the limit of $h$ is not $w$.
Namely, we claim that there exists a $\ttt$-provable $\Sigma[A]$-definable function $h$ from the ordinals into $W'$ such that $h(0)=r$ and for $\alpha>0$:
\begin{equation}
\label{h_def}
 h(\alpha)=
\begin{cases}
    v  & \text { if } \sup_{\sqsubset} \{h(\beta): \beta<\alpha\} \sqsubset v \text { and }\\
       & e(\alpha) \text{ is a pair } \langle v,p \rangle \text{ such that } p \text{ is a witness of}\\
       & \prov(\Godelnum{\forall \gamma\exists \beta>\gamma (h(\beta)\neq v)}) \\
    \sup_{\sqsubset} \{h(\beta): \beta<\alpha\} & \text{ otherwise.}
\end{cases}
\end{equation}



Let us consider the following $\LSigma[A]$ formula, which we define using the Diagonal Lemma. To keep the displayed formula readable, we write $\nu(x)$ for the $\LSigma[A]$-formula
\[\nu(x) \;:\biimp\; \forall \gamma\,\exists \beta>\gamma\,\exists v'\,( G_h(\beta, v')\land v'\neq x),\]
which expresses that $x$ is not the limit of $h$ (compare~(\ref{h_def})). With this abbreviation, the formula reads
\[
\begin{aligned}
 G_h(\alpha,x) :\biimp{} &  (\alpha \not\in \mathsf{Ord} \imp x= r ) \wedge (\alpha=0\imp x=r) \wedge {}\\
   & (\alpha>0\to \exists h_{app} \, (\text{Func}(h_{app}) \wedge \dom(h_{app}) = \alpha \wedge {}\\
   &\quad \forall \beta <\alpha \,\models_\Sigma \Godelnum{G_h(\Godelnum{\beta}, \Godelnum{h_{app} (\beta)})} \wedge {}\\
   &\quad \exists s\in W' \, ( \forall \beta<\alpha \,(h_{app}(\beta)\sqsubseteq s) \wedge \exists \beta<\alpha \,(h_{app}(\beta)=s) \wedge {}\\
   &\qquad (s\sqsubset x \wedge \proov(\pi_1(e(\alpha)), \Godelnum{\nu(\dot x)} ) \wedge \pi_0(e(\alpha))= x ) \vee {}\\
   &\qquad ( s= x \wedge ( \neg (s \sqsubset \pi_0(e(\alpha))) \vee \neg \proov(\pi_1(e(\alpha)), \Godelnum{\nu(\dot x)} ) ) ) ) ))
\end{aligned}
\]

We claim that this formula describes the graph of the function $h$ we are after.
We start by showing that, provably in $\ttt$, $G_h$, when restricted to the ordinals, describes a function to $W'$.
\begin{claim}
    $\ttt$ proves that, for every ordinal $\alpha$, the restriction of $G_h$ to ordinals smaller than $\alpha$ is the graph of a $\sqsubseteq$-weakly ascending function to $W'$.
\end{claim}
\begin{proof}
    We reason in $\ttt$ and prove the claim by transfinite induction on $\alpha$.

    If $\alpha=0$ or $\alpha=1$, the claim follows from the first line of the equivalence describing $G_h$.

    Suppose $\alpha>1$ and the claim holds for all $\beta<\alpha$.
    Notice that if $\alpha$ is a limit, then the claim follows immediately from the inductive assumption. We can thus suppose that $\alpha$ is a successor, say $\alpha=\alpha'+1$. What we have to show is that there exists $x\in W'$ such that $G_h(\alpha',x)$ holds, that there is no $x'\ne x$ for which $G_h(\alpha',x')$ holds, and that $x$ is $\sqsubseteq$-above every $y$ for which $G_h(\beta,y)$ holds for some $\beta<\alpha'$.

    For the sake of readability, we will denote by $h\rst_{\alpha'}$ the function $h:\alpha'\to W'$ given by our inductive assumption. The existence of this function follows from $\Sigma$-Reflection. This $h\rst_{\alpha'}$ will serve as the witness for $h_{app}$ in $G_h(\alpha,x)$.

    We start by showing that $\sup_\sqsubset\{ v\in W': \exists \beta<\alpha'(h\rst_{\alpha'}(\beta)=v)\}$ exists, and is actually a maximum. This maximum will be used as the witness for $s$ in $G_h(\alpha,x)$.

    As we mentioned at the beginning of the proof of the theorem, the presence of the co-ranking function $\mathsf{rk}\colon W'\to\mathsf{Ord}$ and the fact that $h\rst_{\alpha'}$ is $\sqsubseteq$-weakly ascending guarantee that $h\rst_{\alpha'}$ attains the maximum value $m$ from some ordinal $\beta<\alpha'$ onwards.

    Next, we show that we can find an $x$ such that $G_h(\alpha',x)$ holds. First, notice that if we set $h_{app}=h\rst_{\alpha'}$ and $s=m$ in the fixed-point definition of $G_h$, from the inductive assumption and the previous argument we obtain that $\text{Func}(h_{app})$, $ \dom(h_{app}) = \alpha'$, $\forall \beta<\alpha' (h_{app}(\beta)\sqsubseteq s)$ and $\exists \beta<\alpha' (h_{app}(\beta)=s)$ all hold. Moreover, since we know (in $\ttt$) that $G_h(\beta,h_{app}(\beta))$ holds, by \Cref{lem:coded-truth} we can conclude that, for every $\beta<\alpha'$,
    $\models_\Sigma \Godelnum{G_h(\Godelnum{\beta}, \Godelnum{h_{app} (\beta)})}$, and conclude that $\forall \beta <\alpha \models_\Sigma \Godelnum{G_h(\Godelnum{\beta}, \Godelnum{h_{app} (\beta)})}$ by $\Sigma$-Collection.

    We then examine $e(\alpha)$: if $s\sqsubset\pi_0(e(\alpha'))$ and the formula \[\proov(\pi_1(e(\alpha')), \Godelnum{\forall \gamma\exists \beta>\gamma \exists v'( G_h(\beta, v')\land \lnot v'=\Godelnum{\pi_0(e(\alpha'))})} )\]  holds, then we can set $x=\pi_0(e(\alpha'))$. Otherwise, we let $x=s$.

    Next, we prove that the $x$ we selected is unique. Indeed, for every $x'\in W'$, if $G_h(\alpha',x')$ holds, then this is witnessed by some $h_{app}'$ and some $s'$. Hence, $\models_\Sigma \Godelnum{G_h(\Godelnum{\beta}, \Godelnum{h_{app} (\beta)})}$ and $\models_\Sigma \Godelnum{G_h(\Godelnum{\beta}, \Godelnum{h_{app}' (\beta)})}$ both hold, which implies that $h_{app}'= h_{app}$. This in turn easily implies that $s'=s$, and thus that $x'=x$.

    Since the unique $x$ we chose is weakly above $m$, we are done.
\end{proof}

Thanks to the previous claim, we can conclude that $G_h$ defines a $\sqsubseteq$-weakly increasing function from the ordinals into $W'$. We denote by $h$ the function defined by $G_h$. We can now verify by straightforward induction on $\alpha$ that $h$ satisfies the recursive definition (\ref{h_def}). We will henceforth write $h(\alpha)=u$ instead of the more cumbersome $G_h(\alpha,u)$.

Since $\ttt$ is $\Sigma[A]$-sound, we can also make sense of $h$ as a $\Sigma^A$-definable function on $A$ from the perspective of the metatheory. Namely, we define $h^\circ\colon \mathsf{Ord}\cap A \to W'$ by \[h^\circ(\alpha)=v\iff \models_\Sigma^A G_h(\Godelnum{\alpha},\Godelnum{v}).\]
To see that this definition indeed yields a well-defined function, we first observe that for each $\alpha$ there should be at least one $v$ such that $h^\circ(\alpha)=v$, since $\ttt$ is $\Sigma[A]$-sound and proves $\exists v\in W' (h(\Godelnum{\alpha})=v)$. Second, we observe that for each given $\alpha$ there can be at most one $v$ for which $h^\circ(\alpha)=v$, since otherwise, by $\Sigma[A]$-completeness, those two different $v$ would be values of $h(\alpha)$ from the perspective of $\ttt$, leading to a contradiction in $\ttt$, which cannot exist by the $\Sigma[A]$-soundness of $\ttt$.

Since now we know that $\ttt$ proves that $h$ is weakly increasing function to $\mathcal{K}'$ that has co-ranking function, as mentioned in the beginning of the proof of the theorem, we can conclude that $\ttt$ proves that $h$ stabilizes from some point.

For every $u\in W'$ we define the $\langl[A]$ sentence $S_u$ as
$\exists \gamma\forall \alpha(\gamma<\alpha \imp G_h(\alpha, u ))$, which intuitively says that the limit of $h$ is $u$. As in the case
of arithmetic, we define an interpretation $(\cdot)^*$ of $\langl[A]^{\Box}$ into $\langl[A]$ by making $(\cdot)^*$ commute with the Boolean connectives, by setting $(\Box C)^* := \prov (\Godelnum{C^*})$,
and, for every propositional variable $p\in S$,
$p^*:= \bigvee_{j\in \{ i\in W': \mathcal{K}',i\models p \}}S_j$. For all the other propositional variables we set $p^*:= \bigwedge \emptyset$.

Notice that the interpretation $(\cdot)^*$ can be defined inside of $A$ as a $\Sigma^A$-function, since $\mathcal{K}'\in A$. In particular this entails that all the values of $(\cdot)^*$ are elements of $\langl[A]$.

The proof of the fact that $\ttt\not\vdash^A  B^*$ goes through much as in
\cite[Chapter 9]{logic-prov-boolos}. We give an overview of the argument here.

\begin{enumerate}
    \item First of all, 
    since we know that $\ttt$ proves $h$ to be $\sqsubseteq$-weakly increasing and also that it has a limit,
    we can conclude that, for every $u\in W'$,
    $\ttt\vdash^A  \forall \alpha ( h(\alpha)=u\imp (S_u\vee\bigvee_{u\sqsubset v}S_v))$.
    \item Next, we want to show that, for every $u,v\in W'$ with $u\sqsubset v$,
    we have $\ttt\vdash^A  S_u\imp \neg \prov (\Godelnum{\neg S_v})$.
    To do this, we argue within $\ttt$ as follows. If $S_u$ holds, there is an ordinal $\gamma$ such that for every $\beta\geq\gamma$, $h(\beta)=u$ holds.
    Here we use our assumption that, provably in $\ttt$, the surjection $e$ enumerates every element of $A$ cofinally
    often: if there were a $c$ such that
    $\proov(c,\Godelnum{\neg S_v})$ holds for some $v$ with $u\sqsubset v$,
    then there would also be an $\alpha>\gamma$ such that $\pi_0(e(\alpha))=v$ and $\proov(\pi_1(e(\alpha)),\Godelnum{\neg S_v})$ holds;
    but then, by the defining equation of $h$, we would have $h(\alpha+1)=v\neq u$, contradicting $S_u$.
    \item Moreover, if $u\in W'$ with $u\neq r$, then it is easy to see that
    $\ttt\vdash^A  S_u\imp \prov(\Godelnum{\neg S_u})$. Namely, reasoning in $\ttt$ under the assumption $S_u$ we see that there exist the first $\alpha$ such that $h(\alpha)=u$, which in particular means that $\pi_1(e(\alpha))$ is a proof of (an equivalent presentation of) the fact that $v$ is not the limit. 
    \item We claim that, for all $u\in W'$ with $u\neq r$,
    $\ttt\vdash^A  S_u\imp \prov ( \Godelnum{\bigvee_{u\sqsubset v} S_v})$.
    We have already observed that $\ttt\vdash^A  \forall \alpha ( h(\alpha)=u\imp
    (S_u\vee\bigvee_{u\sqsubset v}S_v))$, which implies that
    $\ttt\vdash^A  \prov(\Godelnum{\exists \alpha (h(\alpha)=u)})\imp
    \prov(\Godelnum{S_u\vee\bigvee_{u\sqsubset v}S_v})$ for every 
    $u\in W'$. By the fact that $h$ is $\LSigma[A]$, we also have that
    $\ttt\vdash^A  \exists \alpha (h(\alpha)=u) \imp \prov(\Godelnum{\exists \alpha (h(\alpha)=u)})$,
    from which we deduce that $\ttt\vdash^A  S_u\imp \prov(\Godelnum{\exists \alpha (h(\alpha)=u)})$,
    and thus $\ttt\vdash^A  S_u \imp \prov(\Godelnum{S_u\vee\bigvee_{u\sqsubset v}S_v})$.
    Using now our assumption that $u\neq r$, we can deduce the claim, since we have proved
    that $\ttt\vdash^A  S_u\imp \prov(\Godelnum{\neg S_u})$.
    \item This allows us to prove the fundamental result relating the interpretation $(\cdot)^*$ and the
    model $W'$: namely, for every subsentence $C$ of $B$ and every $u\neq r$,
    if $\mathcal{K}',u\models C$, then $\ttt\vdash^A  S_u\imp C^*$,
    and if $\mathcal{K}',u\not\models C$, then $\ttt\vdash^A  S_u\imp \neg (C^*)$.

    This can be shown by induction on the complexity of the subformula $C$.
    In essence, the only interesting case is when $C$ is $\Box D$.
    Suppose first that $\mathcal{K}',u\models \Box D$.
    This means that for all $v\in W'$ with $u\sqsubset v$, we have $\mathcal{K}',v\models D$,
    and hence $\ttt\vdash^A  S_v\imp D^*$ by the inductive assumption. Therefore, we have
    $\ttt\vdash^A  \bigvee_{u\sqsubset v} S_v \imp D^*$, which entails
    $\ttt\vdash^A  \prov(\Godelnum{\bigvee_{u\sqsubset v} S_v})
    \imp \prov(\Godelnum{D^*})$, and thus, by the previous item, we conclude that
    $\ttt\vdash^A  S_u\imp \prov(\Godelnum{D^*})$.

    If instead $\mathcal{K}',u\not\models \Box D$, then for some $v\in W'$ with
    $u\sqsubset v$ we have $\mathcal{K}',v\not\models D$, hence $\ttt\vdash^A  S_v\imp\neg (D^*) $,
    which entails $\ttt\vdash^A  \neg\prov(\Godelnum{\neg S_v})\imp \neg
    \prov(\Godelnum{D^*}) $. We can then conclude, since we have shown that
    $\ttt\vdash^A  S_u\imp \neg \prov (\Godelnum{\neg S_v})$.
    \item Recall that $t$ is the immediate successor of $r$ in $W'$ (i.e., the root of $W$).
    From the previous result, we obtain that $\ttt\vdash^A  S_t\imp \neg (B^*)$.
    By some algebraic manipulations and the previous results, we get that
    $\ttt\vdash^A  S_r\imp \neg \prov(\Godelnum{B^*})$.
    \item Suppose for a contradiction that $\ttt\vdash^A  B^*$; then
    $\ttt\vdash^A  \prov(\Godelnum{B^*})$, and so $\ttt\vdash^A  \neg S_r$ by the previous item.
    Now notice that $\neg S_r$ can be expressed as a $\Sigma[A]$ formula: indeed, since $h$ is $\sqsubseteq$-weakly increasing and $h(0)=r$, $\neg S_r$ simply says that there exists $\gamma$ such that $h(\gamma)\neq r$.
    Hence, by the $\Sigma[A]$-soundness of $\ttt$, it follows that $\models_\Sigma^A \neg S_r$.
    But if $\models_\Sigma^A \neg S_r$ holds, then for some value $u\neq r$ and some ordinal $\alpha$ we have that $\models_\Sigma^A h(\alpha)=u $.

    We now argue in our metatheory.
    Notice that the statement that $h^\circ$ is not $\sqsubseteq$-weakly increasing is $\Sigma^A$, so if it were true, it would be $\ttt$-provable, which would make $\ttt$ inconsistent. From the fact that $W'$ has a co-ranking function, we can conclude that $h^\circ$ indeed has a limit. Let $v\in W'$ be this limit, and say that for every $\beta\ge \gamma$ we have $h(\beta)=v$, where $\gamma$ is the least ordinal with this property. By the considerations above, it follows that $v\sqsupseteq u \neq r$.
    But by the defining equality of $h$, since $h^\circ$ jumps to $v$ at the ordinal $\gamma$, there is a $\ttt$-proof that $v$ is not the limit of $h$. Thus, since $\ttt$ proves that $h(\gamma)=v$ and that $v$ is not the limit of $h$, it also proves that there is $\delta>\gamma$ for which $h(\delta)\sqsupset v$. The latter is a $\Sigma[A]$-statement, and hence by the $\Sigma[A]$-soundness of $\ttt$ we conclude that there is $\delta>\gamma$ such that $h^\circ(\delta)\sqsupset v$, a contradiction.
\end{enumerate}
\end{proof}


\begin{remark}
    It is known that, over $\mathsf{NBG}$, the existence of a surjection
    from the ordinals onto the universe is equivalent to the Axiom of Global Choice.
    This makes it easy to find examples of $A$ that do not satisfy the assumptions of the Theorem:
    it is sufficient, for instance, to take as $A$ any model of $\mathsf{ZF}$ without choice to
    see that no such surjection can exist, which precludes any $\ttt$ from proving its existence.
    For more on the Axiom of Global Choice, we refer to \cite{found-set-theopfraenkel-bar-hillel-levy}.

   The assumption of the existence of a $\Sigma$-definable surjection from the ordinals onto $V$ in $A$-theories $\ttt$ is of course satisfied for those $A$ that are admissible $L_\alpha$ and $\ttt$ extending $\mathsf{KP}[A]+L=V$. In fact, the assumption of the existence of a $\Sigma[A]$-definable surjection from the $A$-ordinals onto $A$, for an admissible set $A$, is rather close to the assumption that $A$ is a level of $L$. Namely, suppose that $A$ is an admissible set satisfying the axioms that every set is at most countable and that every set is contained in an admissible set. Then, using Shoenfield's Absoluteness Theorem, it is easy to prove that the existence of a $\Sigma[A]$-definable surjection from the $A$-ordinals onto $A$ is equivalent to the fact that $A=L_\alpha[x]$ for some $x\subseteq \omega$.
\end{remark}

\begin{remark}
    We note that the reason why we need to assume the existence of a $\Sigma[A]$ surjection from the ordinals onto $V$ appears to be similar in nature to the unresolved issues with the adaptation of the standard Solovay proof to the case of systems of bounded arithmetic. Note that the question of what the provability logic of systems of bounded arithmetic such as $\mathsf{S}^1_2$ is is a well-known problem in the field of provability logic \citep[Section 5]{problems_beklemishev-visser}. Namely, Solovay's construction essentially relies on considerations of statements about comparisons of smallest G\"odel numbers of proofs of two given sentences. Those comparison sentences must be such that their truth implies their provability within the theory being analyzed. The latter is problematic for systems of bounded arithmetic, since the natural complexity class in which the comparison sentences reside is $\exists \Pi_1^b$, while the provability predicate is only $\exists\Sigma_1^b$-complete. That is, the issue is that the quantifier over all smaller G\"odel numbers of proofs has a range that is too large. The proper analogue of this in the set-theoretic language is the distinction between set-sized range (bounded quantifiers) and quantifiers whose range is a proper class. When we use the surjection from the ordinals onto $V$, we enforce the fact that our enumeration of proofs is such that all initial segments of the enumeration are set-sized.\end{remark}


The previous result allows us to isolate some cases in which the nested sequent calculus
$\dgla$ is complete with respect 
to the semantics provided by the interpretations $(\cdot)^*$.

\begin{definition}
    We say that an ordinal $\alpha$ is \emph{stable} if $L_\alpha$ is
    a $1$-elementary submodel of the constructible universe $L$.
\end{definition}

\begin{remark}\label{rem:stab-ord}
    It is a known fact that for a \emph{countable} stable ordinal $\alpha$,
    $L_\alpha$ is actually $1$-elementarily equivalent to $V$ (this is essentially due to
    Shoenfield's Absoluteness Theorem).
\end{remark}

We are finally able to prove the desired completeness result for $\dgla$. Before doing this, we remark
that for the following result we have to use a stronger metatheory than in the rest of the last two sections,
namely, we assume that Axiom $\beta$ holds.

\begin{theorem}\label{theo:compl-di-syst}
    Suppose Axiom $\beta$ (Mostowski collapsing theorem) holds and $A=L_\alpha$ for a countable stable ordinal $\alpha$, and let $\ttt$ be a $\Sigma[A]$-sound $A$-theory
    extending $\kp[A]+V=L$. Then for every formula $B$ of the modal language $\langl[A]^{\Box}$,
    if $\dgla\not\vdash B$,
    then there is an interpretation $(\cdot)^*$ such that $\ttt\not\vdash^A  B^*$.\footnote{Technically speaking $\langl[A]^\Box$ formulas are not the formulas of the language of $\dgla$, because of a difference in the selection of the base connectives. Thus, when we write $\dgla\not\vdash B$ we mean that $\dgla$ does not prove a singleton sequent consisting of the natural translation of $B$ to the language of the modal Tait calculus.}
\end{theorem}
\begin{proof}
    We first observe that, under our assumption of countability of $A$, the modal
    language $\langl[A]^{\Box}$ is a subset of $\langba$, so the statement of the theorem makes sense.

    Let $B$ be as in the assumptions. Using the countability of $A$, we can use
    \Cref{dgla-completeness} to build a rooted conversely well-founded Kripke model
    $\mathcal{K}$ such that $\mathcal{K}\not\models B$. Notice that, technically speaking, what \Cref{dgla-completeness}
    gives us is a countermodel for the sequent $\{ B\}$, but it is immediate to see that this can
    be regarded as a countermodel to $B$ in the sense we are interested in.

    Thanks to our assumption that Axiom $\beta$ is available in the metatheory, we have that the property of being
    well-founded can be expressed in a $\Sigma_1$ way (i.e., via the existence of a rank function). Moreover, as we have already observed, the relation $\mathcal{K}\models B$ is $\Delta_1$ (with parameters from $A$), and thus so is $\mathcal{K}\not\models B$.
    Hence, the existence of a Kripke model $\mathcal{K}$ as above can be expressed in a $\Sigma_1$ way.
    Since such a $\mathcal{K}$ exists in $V$, using \Cref{rem:stab-ord} we can conclude that there is a $\mathcal{K}'\in A$
    with the same properties.

    As we have already observed, the assumption that $\ttt\vdash^A  V=L$ entails that $\ttt$
    proves the existence of a surjection from the ordinals onto the universe.
    We can then conclude by applying \Cref{theo:km-compl}.
\end{proof}


\section{Appendix: technicalities concerning the development of provability logic
with an admissible class}\label{sec:proof-gandy-classes}

We start by presenting a proof of \Cref{thm:ind-def}.

\begin{proof}

To enhance readability, we will assume that the formula $\varphi$ has one variable $x$. This entails that the predicate letter $S$ is unary. Hence, what we want to find is a $\Sigma^A$ formula $\mathsf{Fix}_\varphi(x)$ such that $\kp + \mathsf{Tr}(A) + \kp^A\vdash \forall x\in A(\mathsf{Fix}_\varphi( x) \biimp \varphi( x,\lambda {x}.\mathsf{Fix}_\varphi( x)))$.

We will use several times the fact that if $S$ occurs only positively in $\varphi(\vec x, S)$, then for any two formulas $R(x)$ and $V(x)$, $\kp + \mathsf{Tr}(A) + \kp^A$ proves  \[\forall x (R(x)\imp V(x))\imp \forall x(\varphi(x,\lambda x. R(x))\imp \varphi(x,\lambda x.V(x))),\]
which can be shown by a straightforward induction on the complexity of $\varphi$.

The main idea of the proof is that we will build the desired fixed point by approximating it from below. To this end, we define the formula $\appr_\varphi(s)$ as
\[
  \text{Func}(s) \wedge \dom(s) \in \mathsf{Ord} \wedge \zeta(s),
\]
where $\text{Func}(s)$ is the $\Delta$ formula asserting that $s$ is a function, and the auxiliary formula $\zeta(s)$ is defined as follows:
\[
  \forall \alpha\in \dom(s) \forall  x\in s(\alpha) ( \varphi(x,\lambda x.x\in\bigcup_{\beta<\alpha}s(\beta)).
\]
Technically speaking, we are abusing notation in the formula above, since if $s$ is not a function then $\dom(s)$ and $s(\beta)$ have no meaning. But since in these cases $\appr_\varphi(s)$ is false anyway, we ignore these trivialities.

Notice that, thanks to the assumption that the predicate $S$ in $\varphi$ occurs positively, substituting $S$ with the formula $x\in \bigcup_{\beta<\alpha}s(\beta)$ produces a $\Sigma^A$ formula with parameter $s$. In turn, this implies that $\appr_\varphi(s)$ is a $\Sigma^A$ formula.

From now on, we will shorten $\varphi(x,\lambda x.x\in\bigcup_{\beta<\alpha}s(\beta))$ to $\varphi(x,\bigcup_{\beta<\alpha}s(\beta))$, for added clarity.

We claim that the $\Sigma^A$ formula $\mathsf{Fix}_\varphi(x)$ defining the desired fixed point is
\[
  \exists s\in A( \appr_\varphi(s) \wedge \exists \alpha \in \dom(s) (x\in s(\alpha))).
\]

We start by showing, in $\kp + \mathsf{Tr}(A) + \kp^A$, that $\fixp_\varphi(x)\biimp \varphi( x,\lambda {x}.\fixp_\varphi( x))$.

We reason in $\kp + \mathsf{Tr}(A) + \kp^A$.
Suppose that $\fixp_\varphi(x)$ holds; then there exists a function $s$ such that $\appr_\varphi(s)$ holds and, for some $\alpha\in \dom(s)$, $x\in s(\alpha)$. Hence, $\varphi(x,\bigcup_{\beta<\alpha}s(\beta))$ holds. Moreover, notice that $y\in\bigcup_{\beta<\alpha}s(\beta)$ trivially implies $\fixp_\varphi(y)$, as witnessed by $s$ itself. Hence, we have that $\fixp_\varphi(x)\imp \varphi( x,\lambda {x}.\fixp_\varphi( x))$, giving the left-to-right implication.

Suppose now that $\varphi( x,\lambda {x}.\fixp_\varphi( x))$ holds. By the principle of $\Sigma^A$-Reflection, available in $\kp^A$, we have that there exists an $a\in A$ such that the relativization $\varphi^a( x,\lambda {x}.\fixp^a_\varphi( x))$ also holds. Let us define $U=\{ s\in a: \appr^a_\varphi(s)\}$, and write $\delta_{s_0}=\bigcup_{s\in U} \dom(s)$ for brevity. Let $s_0^-$ be the function defined as having domain equal to $\delta_{s_0}$, and setting $s_0^-(\alpha)=\bigcup_{s\in U}s(\alpha)$ for every $\alpha\in\dom(s_0^-)$. Finally, we define the function $s_0$ as $s_0^-\cup \{ \langle \delta_{s_0},\{x\}\rangle\}$, i.e.\ $s_0$ maps the ordinal $\delta_{s_0}$ to $x$.

We claim that $s_0$ witnesses $\fixp_\varphi(x)$, which would prove the right-to-left implication. Since obviously $\exists \alpha\in \dom(s_0)(x\in s_0(\alpha))$ holds, we only have to show that $\appr_\varphi(s_0)$ holds too, which boils down to verifying $\zeta(s_0)$, i.e., that $\forall \alpha\in \dom(s_0) \forall  y\in s_0(\alpha) ( \varphi(y,\bigcup_{\beta<\alpha}s_0(\beta))$. We distinguish two cases.
\begin{itemize}
\item Suppose $\alpha<\delta_{s_0}$, and fix $y\in s_0(\alpha)$. Then there is $s\in U$ such that $y\in s(\alpha)$. Since $\appr_\varphi^a(s)$ holds, so does $\appr_\varphi(s)$, and hence $\varphi(y,\bigcup_{\beta<\alpha}s(\beta))$ holds. Since clearly $\bigcup_{\beta<\alpha}s(\beta)\subseteq \bigcup_{\beta<\alpha}s_0(\beta)$, we have that $\varphi(y,\bigcup_{\beta<\alpha}s_0(\beta))$ holds, as we wanted.
\item Suppose instead that $\alpha= \delta_{s_0}$, we need to show that $\varphi(x,\bigcup_{\beta<\alpha}s_0(\beta))$ holds. Since we know that $\varphi^a( x,\lambda {x}.\fixp^a_\varphi( x))$ holds, it suffices to show that, for every $y\in a$, $\fixp_\varphi^a(y)$ implies $y\in \bigcup_{\beta<\alpha}s_0(\beta)$. This follows from the fact that $\fixp_\varphi^a(y)$ implies that there is $s\in U$ with $y\in \ran s$.
\end{itemize}
This argument shows that $\kp + \mathsf{Tr}(A) + \kp^A$ proves $\fixp_\varphi(x)$ to be a fixed point of $\varphi$. 

We are only left to show that $\fixp_\varphi(x)$ describes the minimal upper pre-fixed point of $\varphi$. Suppose not, and let $\theta(x)$ be a formula in the expanded language (i.e., we put no requirements on its complexity, nor on the fact that all the quantifiers are bounded by $A$). Now we reason in $\kp+\mathsf{Tr}(A)+\kp^A$ and assume for a contradiction that $\forall x\in A ( \varphi( x,\lambda {x}.\theta( x))\imp \theta( x))$, yet there is some $x$ such that $\fixp_\varphi(x)\wedge \neg\theta(x)$ holds. By definition, there is a function $s$ such that $\appr_\varphi(s)$ holds and $x\in s(\beta)$ for some $\beta\in \dom(s)$. Using Foundation for formulas in the expanded language, we can find the minimal $\alpha\in\dom(s)$ such that for some $y\in \ran(s)$ we have $\neg\theta(y)$. Hence, $\varphi(y,\bigcup_{\beta<\alpha}s(\beta))$ holds by the definition of $\appr_\varphi$. But we also have that for every $z\in \bigcup_{\beta<\alpha}s(\beta)$, by the minimality of $\alpha$, $\theta(z)$ holds, from which we conclude that $\varphi(y,\lambda y.\theta(y))$ holds, and hence also $\theta(y)$. This gives the desired contradiction.
\end{proof}



Next, to demonstrate how the definitions of infinitary formal languages can be carried out as inductive definitions over $A$, we give a more formal presentation of the definition of the language $\langl[A]$.
\begin{definition}\label{def:langl}
  We fix some  $11$ distinct elements of $\mathbb{HF}$, which we denote by \[\lnot,\to,\forall,\exists,\bigwedge,\bigvee,\mathsf{var},\mathsf{cnst}^{\ulcorner\urcorner},\mathsf{atm},\mathsf{=},\mathsf{\in};\] for the sake of completeness, let us take the finite ordinals $1,\ldots,11$ for this role.
  \begin{enumerate}
      \item The class of $\langl[A]$-variables consists of all the pairs $v_a=\langle \mathsf{var},a\rangle$, where $a\in A$.
      \item The class of $\langl[A]$-constants consists of all the pairs $\Godelnum{a} = \langle \mathsf{cnst}^{\ulcorner\urcorner},a\rangle$.
      \item The class of $\langl[A]$-terms is the union of the classes of constants and variables.
      \item The class of atomic $\langl[A]$-formulas consists of all the tuples $\langle \mathsf{atm},\langle =,\langle t_0,t_1\rangle\rangle \rangle$ (representing $t_0=t_1$) and $\langle \mathsf{atm},\langle \in,\langle t_0,t_1\rangle\rangle \rangle$ (representing $t_0\in t_1$), where $t_0,t_1$ are $\langl[A]$-terms.
  \end{enumerate}

  We let $\langl[A]$ be the smallest class such that:
  \begin{enumerate}
  \item $\langl[A]$ contains all the atomic $\langl[A]$-formulas;
  \item\label{item:lang2} $\varphi\in \langl[A]$ $\Rightarrow$
  $\langle \neg, \varphi\rangle\in \langl[A]$;
  \item\label{item:langl-impl} $\varphi,\psi\in \langl[A]$ $\Rightarrow$
  $\langle \imp, \langle \varphi, \psi\rangle\rangle\in \langl[A]$;
  \item\label{item:lang3} $\varphi\in \langl[A]$ and $x$ is an $\langl[A]$-variable $\Rightarrow$
  $\langle\forall,\langle x, \varphi\rangle \rangle\in \langl[A]$;
   \item\label{item:lang4} $\varphi\in \langl[A]$ and $x$ is an $\langl[A]$-variable $\Rightarrow$
  $\langle\exists, \langle x, \varphi\rangle \rangle\in \langl[A]$;
  \item $\Phi\in A$ and $\Phi\subseteq \langl[A]$ $\Rightarrow$ $\langle \bigwedge, \Phi\rangle\in \langl[A]$;
  \item $\Phi\in A$ and $\Phi\subseteq \langl[A]$ $\Rightarrow$ $\langle \bigvee, \Phi\rangle\in \langl[A]$.
  \end{enumerate}
\end{definition}
Applying Gandy's theorem to the inductive definition of $\langl[A]$ above, we conclude that it gives a $\Sigma^A$ class over $\kp+\mathsf{Tr}(A)+\kp^A$. It is often convenient to furthermore know that $\langl[A]$ is $\Delta^A$. This follows from the fact that a $\Sigma^A$-deciding function for the class can be defined by straightforward transfinite recursion.


\begin{remark}
    In the definition above, we have been quite formal, explicitly using the code of formulas to
    express them. In order to enhance readability, in the following, we will not do so,
    and mostly use standard shortcuts to denote formulas of $\langl[A]$. For instance,
    we will write $\Godelnum{a}\in\Godelnum{b}$ to mean the $\langl[A]$-formula $\langle \mathsf{atm},\langle \in, \langle\langle \mathsf{cnst}^{\ulcorner\urcorner}, a \rangle,\langle \mathsf{cnst}^{\ulcorner\urcorner}, b \rangle\rangle\rangle\rangle$.
\end{remark}




\begin{acknowledgement}
The authors gratefully acknowledge the fruitful discussions with Lev Beklemishev and Philipp Provenzano.

The work of Fedor Pakhomov has been funded by the FWO grant G0F8421N. 
The work of Giovanni Sold\`a has been funded by the FWO senior post-doctoral fellowship 1257725N. 
Mojtaba Mojtahedi is partially funded by the BOF grant BOF.STG.2022.0042.01. Also Mojtaba Mojtahedi was 
partially supported by the Odysseus programme from the Research Foundation-Flanders (FWO), 
project G0ASI25N.

We remark that the authors used Claude Opus 4.8 and GPT 5.5 in the final stages of the writing process, as a reviewing tool: some minor adjustments of the paper are due to it.

\end{acknowledgement}

\end{document}